\documentclass[11pt]{amsart}
\usepackage{amsmath,amssymb,mathrsfs,pstricks}
\usepackage[shortlabels]{enumitem}
\usepackage[colorlinks=true,linkcolor=black,citecolor=black,urlcolor=black]{hyperref}
\usepackage[small,heads=LaTeX,nohug]{diagrams}
\usepackage[all]{xy}

\psset{unit=1pt}
\psset{arrowsize=4pt 1}
\psset{linewidth=.5pt}

\newcommand{\define}{\textbf}

\newcommand{\isom}{\cong}
\renewcommand{\setminus}{\smallsetminus}
\renewcommand{\phi}{\varphi}
\newcommand{\exterior}{\textstyle\bigwedge}
\renewcommand{\tilde}{\widetilde}
\renewcommand{\hat}{\widehat}
\renewcommand{\bar}{\overline}

\renewcommand{\AA}{\mathbb{A}}
\newcommand{\CC}{\mathbb{C}}
\newcommand{\GG}{\mathbb{G}}
\newcommand{\PP}{\mathbb{P}}
\newcommand{\QQ}{\mathbb{Q}}
\newcommand{\RR}{\mathbb{R}}
\newcommand{\ZZ}{\mathbb{Z}}

\newcommand{\shfE}{\mathscr{E}}
\newcommand{\shfF}{\mathscr{F}}

\newcommand{\shfV}{\mathscr{V}}

\newcommand{\catA}{\mathcal{A}}
\newcommand{\catB}{\mathcal{B}}
\newcommand{\catC}{\mathcal{C}}
\newcommand{\catD}{\mathcal{D}}

\newcommand{\Sch}{\mathbf{Sch}}
\newcommand{\dSch}{\mathbf{dSch}}

\newcommand{\Ran}{\mathrm{Ran}}

\newcommand{\fX}{\mathfrak{X}}
\newcommand{\fY}{\mathfrak{Y}}
\newcommand{\fZ}{\mathfrak{Z}}

\newcommand{\ee}{\mathrm{e}}  
\newcommand{\eem}{\varepsilon} 

\newcommand{\OO}{\mathcal{O}} 

\newcommand{\id}{\mathrm{id}}
\newcommand{\pt}{\mathrm{pt}}

\newcommand{\op}{\mathrm{op}}
\newcommand{\opk}{{\op}K}
\newcommand{\kK}{\hat{K}}    
\newcommand{\aA}{\hat{A}}    
\newcommand{\opkK}{{\op}\kK} 

\newcommand{\loc}{\mathrm{loc}} 

\newcommand{\perf}{\mathrm{perf}}

\newcommand{\ch}{\mathrm{ch}} 
\newcommand{\td}{\mathrm{td}} 

\newcommand{\tor}{T\hspace{-.6ex}or}

\DeclareMathOperator{\Hom}{Hom}

\DeclareMathOperator{\Spec}{Spec}
\DeclareMathOperator{\Sym}{Sym}

\DeclareMathOperator{\coh}{coh}
\DeclareMathOperator{\qcoh}{qcoh}

\newtheorem{theorem}{Theorem}[section]
\newtheorem{lemma}[theorem]{Lemma}
\newtheorem{proposition}[theorem]{Proposition}
\newtheorem{corollary}[theorem]{Corollary}
\newtheorem{definition}[theorem]{Definition}

\newtheorem*{cor*}{Corollary}

\theoremstyle{definition}
\newtheorem{remark}[theorem]{Remark}
\newtheorem{example}[theorem]{Example}
\newtheorem{question}[theorem]{Question}
\newtheorem{notation}[theorem]{Notation}

\begin{document}

\title[Equivariant Grothendieck-Riemann-Roch and localization]{Equivariant Grothendieck-Riemann-Roch and localization in operational $K$-theory}

\author{Dave Anderson}
\address{Department of Mathematics, The Ohio State University, Columbus, OH 43210}
\email{anderson.2804@math.osu.edu}

\author{Richard Gonzales}
\address{Department of Sciences, Pontificia Universidad Cat\'olica del Per\'u,
San Miguel, Lima 32, Peru}
\email{rgonzalesv@pucp.edu.pe}

\author[Sam Payne]{Sam Payne \\ With an appendix by Gabriele Vezzosi}
\address{Department of Mathematics, University of Texas, Austin, TX 78712, USA}
\email{sampayne@utexas.edu}

\address{Dipartimento di Matematica ed Informatica, Universit\`{a} di Firenze, Florence, Italy}
\email{gabriele.vezzosi@unifi.it}

\date{May 4, 2020}
\thanks{SP was partially supported by NSF DMS-1702428 and a Simons Fellowship, and completed portions of this work at MSRI. DA was partially supported by NSF DMS-1502201. RG was partially supported by PUCP DGI-2017-1-0130 and PUCP DGI-2018-1-0116}

\begin{abstract}
We produce a Grothendieck transformation from bivariant operational $K$-theory to Chow, with a Riemann-Roch formula that generalizes classical Grothendieck-Verdier-Riemann-Roch.  We also produce Grothen\-dieck transformations and Riemann-Roch formulas that generalize the classical Adams-Riemann-Roch and equivariant localization theorems.  As applications, we exhibit a projective toric variety $X$ whose equivariant $K$-theory of vector bundles does not surject onto its ordinary $K$-theory, and describe the operational $K$-theory of spherical varieties in terms of fixed-point data.

In an appendix, Vezzosi studies operational $K$-theory of derived schemes and constructs a Grothendieck transformation from bivariant algebraic $K$-theory of relatively perfect complexes to bivariant operational $K$-theory.
\end{abstract}

\maketitle

\section{Introduction}

Riemann-Roch theorems lie at the heart of modern intersection theory, and much of modern algebraic geometry.  Grothendieck recast the classical formula for smooth varieties as a functorial property of the Chern character, viewed as a natural transformation of contravariant ring-valued functors, from $K$-theory of vector bundles to Chow theory of cycles modulo rational equivalence, with rational coefficients.  The Chern character does not commute with Gysin pushforward for proper maps, but a precise correction is given in terms of Todd classes, as expressed in the {\it Grothendieck-Riemann-Roch formula}
\[
  f_*( \ch(\xi) \cdot \td(T_X) ) = \ch(f_*\xi)\cdot \td(T_Y),
\]
which holds for any proper morphism $f\colon X \to Y$ of smooth varieties and any class $\xi$ in the Grothendieck group of algebraic vector bundles $K^\circ X$.

For singular varieties, Grothendieck groups of vector bundles do not admit Gysin pushforward for proper maps, and Chow groups of cycles modulo rational equivalence do not have a ring structure.  On the other hand, Baum, Fulton, and MacPherson constructed a transformation $\tau\colon K_\circ X \to A_*(X)_\QQ$, from the Grothendieck group of coherent sheaves to the Chow group of cycles modulo rational equivalence, which satisfies a Verdier-Riemann-Roch formula analogous to the Grothendieck-Riemann-Roch formula, for local complete intersection (l.c.i.) morphisms \cite{bfm, sga6}.  Moreover, Fulton and MacPherson introduced {\em bivariant theories} as a categorical framework for unifying such analogous pairs of formulas.  The prototypical example is a single \emph{Grothendieck transformation} from the bivariant $K$-theory of $f$-perfect complexes to the bivariant operational Chow theory, which simultaneously unifies and generalizes the above Grothendieck-Riemann-Roch and Verdier-Riemann-Roch formulas.

We give a detailed review of bivariant theories in \S\ref{ss.bt}.  For now, recall that a bivariant theory assigns a group $U(f\colon X \to Y)$ to each morphism in a category, and comes equipped with operations of pushforward, along a class of \emph{confined}  morphisms, as well as pullback and product.  It includes a homology theory $U_*$, which is covariant for confined morphisms, and a cohomology theory $U^*$, which is contravariant for all morphisms.   An element $\theta\in U(f\colon X \to Y)$ determines Gysin homomorphisms $\theta^*\colon U_*(Y) \to U_*(X)$ and, when $f$ is confined, $\theta_*\colon U^*(X) \to U^*(Y)$.  An assignment of elements $[f] \in U(f\colon X \to Y)$, for some class of morphisms $f$, is called a {\it canonical orientation} if it respects the bivariant operations.  The Gysin homomorphisms associated to a canonical orientation $[f]$ are often denoted $f^*$ and $f_*$.

If $U$ and $\bar{U}$ are two bivariant theories defined on the same category, a Grothen\-dieck transformation from $U$ to $\bar{U}$ is a collection of homomorphisms $t\colon U(X \to Y) \to \bar{U}(X \to Y)$, one for each morphism, which respects the bivariant operations.  A {\it Riemann-Roch formula}, in the sense of \cite{fm}, is an equality
\[
   t([f]_U) = u_f \cdot [f]_{\bar{U}},
\]
where $u_f\in \bar{U}^*(X)$ plays the role of a generalized Todd class.

In previous work \cite{ap, g.loc}, we introduced a bivariant operational $K$-theory, closely analogous to the bivariant operational Chow theory of Fulton and Mac\-Pherson, which agrees with the $K$-theory of vector bundles for smooth varieties, and developed its basic properties.  Here, we deepen that study by constructing Grothendieck transformations and proving Riemann-Roch formulas that generalize the classical Grothendieck-Verdier-Riemann-Roch, Adams-Riemann-Roch, and Lefschetz-Riemann-Roch, or equivariant localization, theorems.  Throughout, we work equivariantly with respect to a split torus $T$.

\subsection*{Grothendieck-Verdier-Riemann-Roch}
By the equivariant Riemann-Roch theorem of Edidin and Graham, there are natural homomorphisms
\[
  K^T_\circ(X) \to \hat{K}^T_\circ(X)_\QQ \xrightarrow{\tau} \aA^T_*(X)_\QQ,
\]
the second of which is an isomorphism, where the subscript $\QQ$ indicates tensoring with the rational numbers, and $\hat{K}$ and $\aA$ are completions with respect to the augmentation ideal and the filtration by (decreasing) degrees, respectively.  Our first theorem is a bivariant extension of the Edidin-Graham equivariant Riemann-Roch theorem, which provides formulas generalizing the classical Grothendieck-Riemann-Roch and Verdier-Riemann-Roch formulas in the case where $T$ is trivial.

\begin{theorem}\label{t.GRR}
There are Grothendieck transformations
\[
  \opk_T^\circ(X\to Y) \to \opkK_T^\circ(X\to Y)_\QQ \xrightarrow{\ch} \aA_T^*(X\to Y)_\QQ,
\]
the second of which induces isomorphisms of groups, and both are compatible with the natural restriction maps to $T'$-equivariant groups, for $T' \subset T$.

Furthermore, equivariant lci morphisms have canonical orientations, and if $f$ is such a morphism, then
\[
  \ch( [f]_K ) = \td(T_f)\cdot [f]_A,
\]
where $\td(T_f)$ is the Todd class of the virtual tangent bundle.
\end{theorem}

\noindent When $T$ is trivial, and $X$ and $Y$ are quasi-projective, the classical Chern character from algebraic $K$-theory of $f$-perfect complexes to $A^*(X \to Y)$ factors through $\ch$, via the Grothendieck transformation constructed by Vezzosi in Appendix~\ref{app:vezzosi}.  Hence, Theorem~\ref{t.GRR} may be seen as a natural extension of Grothendieck-Verdier-Riemann-Roch.  See also Remark~\ref{r.history}, below.

Specializing the Riemann-Roch formula to statements for homology and cohomology, we obtain the following.
\begin{cor*}
If $f\colon X \to Y$ is an equivariant lci morphism, then the diagrams
\begin{diagram}
 \opk_T^\circ(X)   &\rTo^{\ch} & \aA_T^*(X)_\QQ &  & {K}^T_\circ(X)   &\rTo^{\tau} & \aA^T_*(X)_\QQ \\
 \dTo^{f_*} &   &  \dTo_{f_*(\;\; \cdot \td(T_f))} & \quad  \qquad \text{and} \quad  &  \uTo^{f^*} &   &  \uTo_{\td(T_f)\cdot f^*} \\
 \opk_T^\circ(Y)   & \rTo^{\ch} & \aA_T^*(Y)_\QQ  &  &  {K}^T_\circ(Y)   & \rTo^{\tau} & \aA^T_*(Y)_\QQ
\end{diagram}
commute.  For the first diagram, $f$ is assumed proper.
\end{cor*}

\begin{remark}\label{r.history}
As explained in \cite{fm}, formulas of this type for singular varieties first appeared in \cite{sga6} and \cite{verdier}, respectively; a homomorphism like $\tau$, taking values in (non-equivariant) singular homology groups, was originally constructed in \cite{bfm}.  The homomorphism $\tau$ was first constructed for equivariant theories by Edidin and Graham \cite{eg-rr}, with the additional hypothesis that $X$ and $Y$ be equivariantly embeddable in smooth schemes.  A more detailed account of the history of Riemann-Roch formulas can be found in \cite[\S18]{fulton-it}.

These earlier Grothendieck transformations and Riemann-Roch formulas all take some version of algebraic or topological $K$-theory as the source, and typically carry additional hypotheses, such as quasi-projectivity or embeddability in smooth schemes.  For instance, for quasi-projective schemes, Fulton gives a Grothendieck transformation $K^\circ_\perf(X\to Y) \to A^*(X \to Y)_\QQ$ which, by construction, factors through $\opk^\circ(X\to Y)$ \cite[Ex.~18.3.16]{fulton-it}.  Combining Theorem~\ref{t.GRR} with Vezzosi's Theorem~\ref{thm:Groth-transform}, which gives a Grothendieck transformation $K^\circ_\perf(X \to Y) \to \opk^\circ(X \to Y)$, we see that Fulton's Grothendieck transformation extends to arbitrary schemes.

Other variations of bivariant Riemann-Roch theorems have been studied for topological and higher algebraic $K$-theory; see, e.g., \cite{williams,levy}.
\end{remark}

\begin{remark}
Vezzosi's proof of Theorem~\ref{thm:Groth-transform} uses derived algebraic geometry in an essential way.  It seems difficult to prove the existence of such a Grothendieck transformation directly, in the category of ordinary (underived) schemes.
\end{remark}

\subsection*{Adams-Riemann-Roch}

Our second theorem is an extension of the classical Adams-Riemann-Roch theorem.  Here, the role of the Todd class is played by the equivariant Bott elements $\theta^j$, which are invertible in $\opkK_T^\circ(X)[{j}^{-1}]$.

\begin{theorem}\label{t.ARR}
There are Grothendieck transformations
\[
  \opk_T^\circ(X\to Y) \xrightarrow{\psi^j} \opkK_T^\circ(X\to Y)[{j}^{-1}],
\]
for each nonnegative integer $j$, that specialize to the usual Adams operations $\psi^j\colon K_T^\circ X \to K_T^\circ X$ when $X$ is smooth.

There is a Riemann-Roch formula
\[
  \psi^j([f]) = \theta^j(T_f^\vee)^{-1}\cdot [f],
\]
for an equivariant lci morphism $f$.
\end{theorem}

As before, the Riemann-Roch formula has the following specializations.
\begin{cor*}
If $f\colon X\to Y$ is an equivariant lci morphism, the diagrams
\begin{diagram}
 \opk_T^\circ(X)   &\rTo^{\psi^j} & \opkK_T^\circ(X)[j^{-1}] &  & {K}^T_\circ(X)   &\rTo^{\psi_j} & \hat{K}^T_\circ(X)[j^{-1}] \\
 \dTo^{f_*} &   &  \dTo_{f_*(\;\; \cdot \theta^j(T_f^\vee)^{-1})} & \quad  \qquad \text{and} \quad  &  \uTo^{f^*} &   &  \uTo_{\theta^j(T_f^\vee)^{-1}\cdot f^*} \\
 \opk_T^\circ(Y)   & \rTo^{\psi^j} & \opkK_T^\circ(Y)[j^{-1}]  &  &  {K}^T_\circ(Y)   & \rTo^{\psi_j} & \hat{K}^T_\circ(Y)[j^{-1}]
\end{diagram}
commute, where $f$ is also assumed to be proper for the first diagram.

\end{cor*}

\noindent
In particular, for $f$ proper lci and a class $c \in \opk^\circ_T(X)$, we have
\[
\psi^j f_* (c) = f_*( \theta^j(T_f^\vee)^{-1} \cdot \psi^j(c)),
\]
in $\opkK^\circ_T(Y)[j^{-1}]$.  This generalizes the equivariant Adams-Riemann-Roch formula for projective lci morphisms from \cite{Kock98}.

\subsection*{Lefschetz-Riemann-Roch}

Localization theorems bear a striking formal resemblance to Riemann-Roch theorems, as indicated in the Lefschetz-Riemann-Roch theorem of Baum, Fulton, and Quart \cite{bfq}.   Our third main theorem makes this explict: we construct Grothendieck transformations from operational equivariant $K$-theory (resp. Chow theory) of $T$-varieties to operational equivariant $K$-theory (resp. Chow theory) of their $T$-fixed loci.

Our Riemann-Roch formulas in this context are generalizations of classical localization statements, in which {\em equivariant multiplicities} play a role analogous to that of Todd classes in Grothendieck-Riemann-Roch.  To define these equivariant multiplicities, one must invert some elements of the base ring.

 Let $M=\Hom(T,\mathbb{G}_m)$ be the character group, so $K_T^\circ(\pt)=\ZZ[M]=R(T)$, and $A_T^*(\pt)=\Sym^*M=:\Lambda_T$.  Let $S\subseteq R(T)$ be the multiplicative set generated by $1-\ee^{-\lambda}$, and let $\bar{S}\subseteq \Sym^*M=A_T^*(\pt)$ be generated by $\lambda$, as $\lambda$ ranges over all nonzero characters in $M$.

\begin{theorem}\label{t.LRR}
There are Grothendieck transformations
\begin{align*}
 \opk_T^\circ(X\to Y) \to S^{-1}\opk_T^\circ(X\to Y) &\xrightarrow{\loc^K}  S^{-1} \opk_T^\circ(X^T \to Y^T)  
 \quad \text{and}\\
 A_T^*(X\to Y)  \to  \bar{S}^{-1} A_T^*(X \to Y) &\xrightarrow{\loc^A} \bar{S}^{-1} A_T^*(X^T \to Y^T) ,
\end{align*}
inducing isomorphisms of $S^{-1}R(T)$-modules and $\bar{S}^{-1}\Lambda_T$-modules, respectively.

Furthermore, if $f\colon X \to Y$  is a flat equivariant map whose restriction to fixed loci $f^T\colon X^T \to Y^T$ is smooth, then there are {\em equivariant multiplicities} $\eem^K(f)$ in $S^{-1}\opk^\circ_T(X^T)$ and $\eem^A(f)$ in $\bar{S}^{-1}A_T^*(X^T)$, so that
\begin{align*}
  \loc^K([f]) &= \eem^K(f)\cdot [f^T] \\ \intertext{and}
   \loc^A([f]) &= \eem^A(f)\cdot [f^T].
\end{align*}
\end{theorem}

\begin{cor*}
Let $f\colon X \to Y$ be a flat equivariant morphism whose restriction to fixed loci $f^T\colon X^T \to Y^T$ is smooth.  Then the diagrams
\begin{diagram}
 \opk_T^\circ(X)   &\rTo & S^{-1}\opk_T^\circ(X^T) &  & K^T_\circ(X)   &\rTo  & S^{-1}K^T_\circ(X^T) \\
 \dTo^{f_*} &   &  \dTo_{f^T_*(\;\; \cdot \eem^K(f))} \quad &  \qquad \text{and} \qquad &  \uTo^{f^*} &   &  \uTo_{\eem^K(f)\cdot (f^T)^*} \\
 \opk_T^\circ(Y)   & \rTo  & S^{-1}\opk_T^\circ(Y^T)  &  &  K^T_\circ(Y)   & \rTo & S^{-1}K^T_\circ(Y^T)
\end{diagram}
commute, where $f$ is assumed proper for the first diagram.  Under the same conditions, the following diagrams also commute:
\begin{diagram}
 A_T^*(X)   &\rTo & \bar{S}^{-1}A_T^*(X^T) &  & A^T_*(X)   &\rTo  & \bar{S}^{-1}A^T_*(X^T) \\
 \dTo^{f_*} &   &  \dTo_{f^T_*(\;\; \cdot \eem^A(f))} \quad & \quad \qquad \text{and} \qquad &  \uTo^{f^*} &   &  \uTo_{\eem^A(f)\cdot (f^T)^*} \\
 A_T^*(Y)   & \rTo  & \bar{S}^{-1}A_T^*(Y^T)  &  &  A^T_*(Y)   & \rTo & \bar{S}^{-1}A^T_*(Y^T).
\end{diagram}
\end{cor*}

In the case where $Y=\pt$ and $X^T$ is finite\footnote{More precisely, the fixed points should be {\it nondegenerate}, a condition which guarantees the scheme-theoretic fixed locus is reduced.}, the first diagram of the Corollary provides an Atiyah-Bott type formula for the equivariant Euler characteristic (or integral, in the case of Chow).  If in addition $X$ is smooth, then this is precisely the Atiyah-Bott-Berline-Vergne formula: for $\xi\in K_T^\circ(X)$ and $\alpha\in A_T^*X$,
\begin{align*}
  \chi(\xi) &= \sum_{p\in X^T} \frac{\xi|_p}{(1-\ee^{-\lambda_1(p)})\cdots(1-\ee^{-\lambda_n(p)})}  \intertext{ and } \int_X \alpha &= \sum_{p\in X^T} \frac{\alpha|_p}{\lambda_1(p) \cdots \lambda_n(p)},
\end{align*}
where $\lambda_1(p),\ldots,\lambda_n(p)$ are the weights of $T$ acting on $T_pX$.

\medskip

These three Riemann-Roch theorems are compatible with each other, as explained in the statements of Theorems~\ref{t.rr}, \ref{t.arr}, and \ref{t.local}.  This compatibility includes localization formulas for Todd classes and Bott elements.  For instance, if $X \to \pt$ is lci and $p\in X^T$ is a nondegenerate fixed point, then
\[
  \td(X)|_p = \frac{\ch(\eem^K_p(X))}{\eem^A_p(X)} \qquad \text{ and } \qquad \theta^j(X)|_p = \frac{\eem^K_p(X)}{\psi^j(\eem^K_p(X))}.
\]
When $p\in X^T$ is nonsingular, we recover familiar expressions for these classes.  Indeed, suppose the weights for $T$ acting on $T_pX$ are $\lambda_1(p),\ldots,\lambda_n(p)$, as above.  Then the formulas for the Todd class and Bott element become
\[
  \td(X)|_p = \prod_{i=1}^n\frac{\lambda_i(p)}{1-\ee^{-\lambda_i(p)}} \qquad \text{ and } \qquad
  \theta^j(X)|_p = \prod_{i=1}^n\frac{1-\ee^{-j\cdot\lambda_i(p)}}{1-\ee^{-\lambda_i(p)}}
\]
See Remark~\ref{r.todd} for more details.

\begin{remark}
The problem of constructing Grothendieck transformations extending given tranformations of homology or cohomology functors was posed by Fulton and MacPherson.  Some general results in this direction were given by Brasselet, Sch\"urmann, and Yokura \cite{bsy}.  They do consider operational bivariant theories, but do not require operators to commute with refined Gysin maps and, consequently, do not have Poincar\'e isomorphisms for smooth schemes.
\end{remark}

\subsection*{Applications to classical $K$-theory}

Mer\-kurjev studied the restriction maps, from $G$-equivariant $K$-theory of vector bundles and coherent sheaves to ordinary, non-equivariant $K$-theory, for various groups $G$.  Notably, he showed that the restriction map for $T$-equivariant $K$-theory of coherent sheaves is always surjective, which raises the question of when this also holds for vector bundles \cite{merkurjev97,merkurjev}.  In Section~\ref{sec:toric}, as one application of our Riemann-Roch and localization theorems, we give a negative answer for toric varieties.

\begin{theorem}
There are projective toric threefolds $X$ such that the restriction map from the $K$-theory of $T$-equivariant vector bundles on $X$ to the ordinary $K$-theory of vector bundles on $X$ is not surjective.
\end{theorem}

As a second application of our main theorems, in Section~\ref{ss:spherical}, we use localization to completely describe the equivariant operational $K$-theory of arbitrary spherical varieties in terms of fixed point data.
Our description is independent of recent results by  Banerjee and Can on
smooth spherical varieties \cite{ekt-sph}.

Some of these results were announced in \cite{anderson}.

\medskip

\noindent {\bf Acknowledgments.} We thank Michel Brion, William Fulton, Jos\'e Gon\-z\'alez, Gabriele Vezzosi, and Charles Weibel for helpful comments and conversations related to this project.

\section{Background on operational $K$-theory}

We work over a fixed ground field, which we assume to have characteristic zero in order to use resolution of singularities.  All schemes are separated and finite type, and all tori are split over the ground field.

\subsection{Equivariant $K$-theory and Chow groups}

Let $T$ be a torus, and let $M = \Hom(T,\GG_m)$ be its character group.  The representation ring $R(T)$ is naturally identified with the group ring $\ZZ[M]$, and we write both as $\bigoplus_{\lambda\in M} \ZZ\cdot \ee^\lambda$.

For a $T$-scheme $X$, let $K^T_\circ(X)$ and $K_T^\circ(X)$ be the Grothendieck groups of $T$-equivariant coherent sheaves and $T$-equivariant perfect complexes on $X$, respecitvely.  We write $A^T_*(X)$ and $A^*_T(X)$ for the equivariant Chow homology and equivariant operational Chow cohomology of $X$.  There are natural identifications
\[
  R(T) = K_T^\circ(\pt) = K^T_\circ(\pt) = \ZZ[M]
\]
and
\[
  \Lambda_T:= A_T^*(\pt) = A^T_*(\pt) = \Sym^*M.
\]
Choosing a basis $u_1,\ldots,u_n$ for $M$, we have $R(T) = \ZZ[\ee^{\pm u_1},\ldots,\ee^{\pm u_n}]$ and $\Lambda_T = \ZZ[u_1,\ldots,u_n]$.

A crucial fact is that both $K^T_\circ$ and $A^T_*$ satisfy a certain descent property.  An {\em equivariant envelope} is a proper $T$-equivariant map $X' \to X$ such that every $T$-invariant subvariety of $X$ is the birational image of some $T$-invariant subvariety of $X'$.  When $X' \to X$ is an equivariant envelope, there are exact sequences
\begin{align}
 A^T_*(X'\times_X X') \to A^T_*(X') \to A^T_*(X) \to 0 \label{e.chow-descent} \\
\intertext{and}
 K^T_\circ(X'\times_X X') \to K^T_\circ(X') \to K^T_\circ(X) \to 0 \label{e.k-descent}
\end{align}
of $\Lambda_T$-modules and $R(T)$-modules, respectively.
The Chow sequence admits an elementary proof (see \cite{kimura,payne-chow}); the sequence for $K$-theory seems to require more advanced techniques (\cite{gillet,ap}).

\subsection{Bivariant theories}\label{ss.bt}

We review some foundational notions on bivariant theories from \cite{fm} (see also \cite[\S4]{ap} or \cite{gk}).  Consider a category $\catC$ with a final object $\pt$, equipped with distinguished classes of \emph{confined} morphisms and \emph{independent} commutative squares.  A bivariant theory assigns a group $U(f\colon X \to Y)$ to each morphism in $\catC$, together with homomorphisms
\begin{align*}
  \cdot \colon U(X\xrightarrow{f} Y) \otimes U( Y \xrightarrow{g} Z)  &\to U( X \xrightarrow{g\circ f} Z) & & \text{(product)}, \\
 f_*\colon U(X \xrightarrow{h} Z ) & \to U(Y \to Z) & &\text{(pushforward)} , \text{ and} \\
 g^*\colon U(X \xrightarrow{f} Y) & \to U(X' \xrightarrow{f'} Y') & &\text{(pullback)},
\end{align*}
where for pushforward, $f\colon X \to Y$ is confined, and for pullback, the square
\begin{diagram}
 X' & \rTo^{f'} & Y' \\
 \dTo & &  \dTo_g \\
 X & \rTo^f & Y
\end{diagram}
is independent.  This data is required to satisfy axioms specifying compatibility with product, for composable morphisms, pushforward along confined morphisms, and pullback across independent squares.

Any bivariant theory determines a \emph{homology} theory $U_*(X) = U(X \to \pt)$, which is covariant for confined morphisms, and a \emph{cohomology} theory $U^*(X) = U(\mathrm{id}\colon X \to X)$, which is contravariant for all morphisms.  An element $\alpha$ of $U(f\colon X\to Y)$ determines a Gysin map $f^\alpha \colon U_*(Y) \to U_*(X)$, sending $\beta\in U_*(Y) = U(Y \to \pt)$ to $\alpha\cdot\beta \in U(X \to \pt) = U_*(X)$.  Similarly, if $f$ is confined, $\alpha$ determines a Gysin map $f_\alpha\colon U^*(X) \to U^*(Y)$, sending $\beta\in U^*(X) = U(X \to X)$ to $f_*(\beta\cdot \alpha) \in U( Y \to Y ) = U^*(Y)$.  A {\em canonical orientation} for a class of composable morphisms is a choice of elements $[f]\in U(f\colon X\to Y)$, one for each $f$ in the class, which respects product for compositions, with $[\id]=1$.  The Gysin maps determined by $[f]$ are denoted $f^!$ and $f_!$.

\subsection{Operational Chow theory and $K$-theory}\label{ss.chow-K}

As described above, a bivariant theory $U$ determines a homology theory.  Conversely, starting with any homology theory $U_*$, one can build an {\em operational} bivariant theory $\op U$, with $U_*$ as its homology theory, by defining elements of $\op U(X \to Y)$ to be collections of homomorphisms $U_*(Y') \to U_*(X')$, one for each morphism $Y' \to Y$ (with $X'=X\times_Y Y'$), subject to compatibility with pullback and pushforward.

We focus on the operational bivariant theories associated to equivariant $K$-theory of coherent sheaves $K^T_\circ(X)$ and Chow homology $A^T_*(X)$.  The category $\catC$ is $T$-schemes, confined morphisms are equivariant proper maps, and all fiber squares are independent.  Operators are required to commute with proper pushforwards and refined pullbacks for flat maps and regular embeddings. 

The basic properties of $A_T^*(X \to Y)$ can be found in \cite{fm,fulton-it,kimura,eg-eit}, and those of $\opk_T^\circ(X \to Y)$ are developed in \cite{ap,g.loc}.  The following properties are most important for our purposes.  We state them for $K$-theory, but the analogous statements also hold for Chow.

\begin{enumerate}[(a)]

\item Certain morphisms $f\colon X \to Y$, including regular embeddings and flat morphisms, come with a distinguished {\em orientation} class $[f] \in \opk_T^\circ(X \to Y)$, corresponding to refined pullback.  When both $X$ and $Y$ are smooth, an arbitrary morphism $f \colon X \to Y$ has an orientation class $[f]$, obtained by composing the classes of the graph $\gamma_f\colon X \to X \times Y$ (a regular embedding) with that of the (flat) projection $p\colon X \times Y \to Y$.

\medskip
\item For any $X$, there is a homomorphism from $K$-theory of perfect complexes to the contravariant operational $K$-theory, $K_T^\circ(X) \to \opk_T^\circ(X)$; there is also a canonical isomorphism $\opk_T^\circ(X\to\pt) \to K^T_\circ(X)$.

\medskip
\item  If $f\colon X \to Y$ is any morphism, and $g\colon Y\to Z$ is smooth, then there is a canonical {\em Poincar\'e isomorphism} $\opk_T^\circ(X\to Y) \to \opk_T^\circ(X \to Z)$, given by product with $[g]$.

\medskip
\item Combining the above, there are homomorphisms
\[
  K_T^\circ(X) \to \opk_T^\circ(X) \to K^T_\circ(X),
\]
which are isomorphisms when $X$ is smooth.

\end{enumerate}

The main tools for computing operational $K$ groups and Chow groups are the following two {\em Kimura sequences}, whose exactness is proved  for $K$-theory in \cite[Propositions~5.3 and 5.4]{ap} and for Chow theory in \cite[Theorems~2.3 and 3.1]{kimura}.  We continue to state only the $K$-theory versions.  First, suppose $Y' \to Y$ is an equivariant envelope, and let $X' = X \times_Y Y'$.  Then
\begin{equation}\label{e.kimura1}
  0 \to \opk_T^\circ (X\to Y) \to \opk_T^\circ (X'\to Y') \to \opk_T^\circ( X' \times_X X' \to Y' \times_Y Y' )
\end{equation}
is exact.  This is, roughly speaking, dual to the descent sequence \eqref{e.k-descent}.

Next, suppose $p\colon Y' \to Y$ is furthermore birational, inducing an isomorphism $Y'\setminus E \xrightarrow{\sim} Y \setminus B$ (where $E = f^{-1}B$).  Given $f\colon X \to Y$, define $A = f^{-1}B \subseteq X$ and $D = f'^{-1}E \subseteq X'$.  Then
\begin{equation}\label{e.kimura2}
  0 \to \opk_T^\circ (X \to Y) \to \opk_T^\circ( X' \to Y') \oplus \opk_T^\circ(A \to B) \to \opk_T^\circ (D \to E )
\end{equation}
is exact. (Only the contravariant part of this sequence is stated explicitly in \cite{ap}, but the proof of the full bivariant version is analogous, following \cite{kimura}.)

\begin{remark}\label{r.exact}
Exactness of the sequences \eqref{e.kimura1} and \eqref{e.kimura2} follow from exactness of the descent sequence \eqref{e.k-descent}.  Hence, if one applies an exact functor of $R(T)$-modules to $K_\circ^T$ before forming the operational bivariant theory, then the analogues of \eqref{e.kimura1} and \eqref{e.kimura2} are still exact.  
For example, given a multiplicative set $S\subseteq R(T)$, the Kimura sequences for $\mathrm{op}S^{-1}K_T^\circ$ are exact.
\end{remark}

\subsection{Kan extension}

By resolving singularities, the second Kimura sequence implies an alternative characterization of operational Chow theory and $K$-theory: they are {\em Kan extensions} of more familiar functors on smooth schemes.  This is a fundamental construction in category theory; see, e.g., \cite[\S X]{cwm}.

Suppose we have functors $I\colon \catA \to \catB$ and $F\colon \catA \to \catC$.  A right Kan extension of $F$ along $I$ is a functor $R=\Ran_I(F)\colon \catB \to \catC$ and a natural transformation $\gamma\colon R\circ I \Rightarrow F$, which is universal among such data: given any other functor $G\colon \catB \to \catC$ with a transformation $\delta\colon G\circ I \Rightarrow F$, there is a unique transformation $\eta \colon G \Rightarrow R$ so that the diagram
\begin{diagram}
 G\circ I &  & \rTo^\eta &  & R \circ I \\
     & \rdTo_\delta &     & \ldTo_\gamma \\
     &       &  F
\end{diagram}
commutes.  
The proof of the following lemma is an exercise.

\begin{lemma}\label{l.kan}
With notation as above, suppose that $F$ admits a right Kan extension $(R,\gamma)$ along $I$.  Assume $\gamma$ is a natural isomorphism.  Then if $T\colon \catC \to \catD$ is any functor, the composite $T\circ F$ admits a Kan extension along $I$, and there is a natural isomorphism
\[
  \Ran_I(T\circ F) \isom T\circ \Ran_I(F).
\]
\end{lemma}

\noindent By \cite[Corollary X.3.3]{cwm}, the hypothesis that $\gamma$ be a natural isomorphism is satisfied whenever the functor $I\colon \catA \to \catB$ is fully faithful.  

For the embedding $I\colon (T\text{-}\mathbf{Sm})^{\mathrm{op}} \to (T\text{-}\mathbf{Sch})^{\mathrm{op}}$ of smooth $T$-schemes in all $T$-schemes, \cite[Theorem~5.8]{ap} shows that the contravariant functor $\opk_T^\circ$ is the right Kan extension of $K_T^\circ$.  
Similarly, operational Chow cohomology is the right Kan extension of the intersection ring on smooth schemes.  
Analogous properties hold for the full bivariant theories, with the same proofs, as we now explain.

Let $\catB'$ be the category whose objects are equivariant morphisms of $T$-schemes $X \to Y$; a morphism $f\colon (X'\to Y') \to (X \to Y)$ is a fiber square
\begin{diagram}
  X' & \rTo^{f'} & X \\
  \dTo &  & \dTo \\
  Y' & \rTo^f & Y.
\end{diagram}
Let $\catA'$ be the same, but where the objects are $X\to Y$ with $Y$ smooth.  Let $\catA=(\catA')^{\mathrm{op}}$ and $\catB=(\catB')^{\mathrm{op}}$, and let $I\colon\catA \to \catB$ be the evident embedding.  The functor $F\colon \catA \to (R(T)\text{-}\mathbf{Mod})$ is given on objects by $F(X\to Y) = K^T_\circ(X)$.  To a morphism $(X'\to Y') \to (X \to Y)$, the functor assigns the refined pullback $f^!\colon K^T_\circ(X) \to K^T_\circ(X')$.  Explicitly, for a sheaf $\shfF$ on $X$, we have $f^![\shfF] = \sum (-1)^i [\tor_i^Y(\OO_{Y'},\shfF)]$, which is well-defined since $f$ has finite Tor-dimension.

\begin{proposition} \label{p.kan}
With notation as above, operational bivariant $K$-theory is the right Kan extension of $F$ along $I$.
\end{proposition}

\begin{proof}
Just as in \cite[Theorem~5.8]{ap}, one applies the Kimura sequence \eqref{e.kimura2}, together with induction on dimension, to produce a natural homomorphism $G(X \to Y) \to \opk_T^\circ(X \to Y)$ for any functor $G$ whose restriction to smooth schemes has a natural transformation to $F$.
\end{proof}

Since the only input in proving the proposition is the Kimura sequence, a similar statement holds if one applies an exact functor of $R(T)$-modules, as pointed out in Remark~\ref{r.exact}.

\begin{lemma}\label{l.order}
Let $S\subseteq R(T)$ be a multiplicative set.  There is a canonical isomorphism of functors
\[
  S^{-1}\opk_T^\circ(X \to Y) \isom \mathrm{op}S^{-1}K_T^\circ(X \to Y),
\]
where the right-hand side is the operational theory associated to $S^{-1}K^T_\circ(X)$.

Similarly, let $J\subseteq R(T)$ be an ideal, and let $\hat{(-)}$ denote $J$-adic completion of an $R(T)$-module.  There is a canonical isomorphism of functors
\[
   \hat{\opk_T^\circ(X \to Y)} \isom \opkK_T^\circ(X \to Y),
\]
where the right-hand side is the operational theory associated to $\hat{K}^T_\circ(X)$.
\end{lemma}

\begin{proof}
Since localization and completion are exact functors of $R(T)$-modules, the right-hand sides satisfy the Kimura sequences and are therefore Kan extensions, as in Proposition~\ref{p.kan}.  The statements now follow from Lemma~\ref{l.kan}.
\end{proof}

A common special case of the first isomorphism is tensoring by $\QQ$, so we will use abbreviated notation: for any $R(T)$-module $B$, we let $B_\QQ = B \otimes_\ZZ \QQ$, and write $\opk^\circ_{T}(X\to Y)_\QQ$ for the bivariant theory associated to $K^T_\circ(X)_\QQ$.

While localization and completion do not commute in general, they do in the main case of interest to us: the completion of $R(T)$ along the augmentation ideal, and the localization given by $\otimes\QQ$.  Thus we may write $\hat{K}^T_\circ(X)_\QQ$ unambiguously, and we write $\opkK^\circ_T(X\to Y)_\QQ$ for the associated operational bivariant theory.

\begin{remark}
The standing hypotheses of characteristic zero is made chiefly to be able to use resolution of singularities in proving the above results.  When using $\QQ$-coefficients, it is tempting to appeal to de Jong's alterations to prove an analogue of the Kimura sequence.  However, if $X' \to X$ is an alteration, with $X'$ smooth, and $X'\setminus E \to X \setminus S$ \'etale, we do not know whether the sequence
\[
  0 \to \opk^\circ(X)_\QQ \to \opk^\circ(X')_\QQ \oplus \opk^\circ(S)_\QQ \to \opk^\circ(E)_\QQ
\]
is exact.  For special classes of varieties that admit smooth equivariant envelopes, our arguments work in arbitrary characteristic.  The special case of toric varieties is treated in \cite{ap}.  In Section~\ref{ss:spherical}, we carry out analogous computations more generally, for spherical varieties.
\end{remark}

\begin{remark}\label{r.weaker}
The proofs of the Poincar\'e isomorphisms (\cite[Proposition~17.4.2]{fulton-it} and \cite[Proposition~4.3]{ap}) only require commutativity of operations with pullbacks for regular embeddings and smooth morphisms.  If one defines operational bivariant theories replacing the axiom of commutativity with flat pullback with the {\it a priori} weaker axiom of commutativity with smooth pullback, the Kan extension properties of $A_T^*$ and $\opk_T^\circ$ show that the result is the same.
\end{remark}

\subsection{Grothendieck transformations and Riemann-Roch}\label{ss.gt}

As motivation and context for the proofs in the following sections, we review the bivariant approach to Riemann-Roch formulas via canonical orientations, following \cite{fm}.

We return to the notation of \S\ref{ss.bt}, so $\catC$ is a category with a final object and distinguished classes of confined morphisms and independent squares, and $U$ is a bivariant theory on $\catC$.  A class of morphisms in $\catC$ carries {\em canonical orientations} for $U$ if, for each $f\colon X\to Y$ in the class, there is $[f]_U\in U(X\to Y)$, such that
\begin{enumerate}[(i)]
\item for $X\xrightarrow{f} Y \xrightarrow{g} Z$, $[f]_U\cdot[g]_U = [gf]_U$ in $U(X\to Z)$; and

\medskip
\item $[\id_X]_U = 1$ in $U^*(X)$.

\medskip
\end{enumerate}
We omit the subscript and simply write $[f]$ when the bivariant theory is understood. In $K^\circ_\perf(X\to Y)$, proper flat morphisms have canonical orientations given by $[f]=[\OO_X]$.  A canonical orientation $[f]$ determines functorial Gysin homomorphisms $f^!\colon U_*(Y) \to U_*(X)$ and, if $f$ is confined, $f_!\colon U^*(X)\to U^*(Y)$.

Now consider another category $\bar\catC$ with a bivariant theory $\bar{U}$.  Let $F\colon \catC \to \bar\catC$ be a functor preserving final objects, confined morphisms, and independent squares.  We generally write $X$, $f$, etc., for objects and morphisms of $\catC$, and $\bar{X}$, $\bar{f}$, etc., for those of $\bar\catC$.  When no confusion seems likely, we sometimes abbreviate the functor $F$ by writing $\bar{X}$ and $\bar{f}$ for the images under $F$ of an object $X$ and morphism $f$, respectively.
A {\em Grothendieck transformation} is a natural map $U(X \to Y) \to \bar{U}(\bar{X} \to \bar{Y})$, compatible with product, pullback, and pushforward.

In the language of \cite{fm}, a {\em Riemann-Roch formula} for a Grothendieck transformation $t\colon U(X\to Y) \to \bar{U}(\bar{X}\to\bar{Y})$ is an equation
\[
  t([f]_U) = u_f\cdot [\bar{f}]_{\bar{U}},
\]
for some $u_f\in \bar{U}^*(\bar{X})$.  For the homology and cohomology components, this translates into commutativity of the diagrams
\begin{diagram}
 U^*(X)   &\rTo^{t^.} & \bar{U}^*(\bar{X}) \\
 \dTo^{f_!} &   &  \dTo_{\bar{f}_!(\;\; \cdot u_f)} \\
 U^*(Y)   & \rTo^{t^.} & \bar{U}^*(\bar{Y})
\end{diagram}
and
\begin{diagram}
 U_*(Y)   &\rTo^{t_.} & \bar{U}_*(\bar{Y}) \\
 \dTo^{f^!} &   &  \dTo_{u_f\cdot \bar{f}^!} \\
 U_*(X)   & \rTo^{t_.} & \bar{U}_*(\bar{X}).
\end{diagram}

Our focus will be on operational bivariant theories built from homology theories, with the operational Chow and $K$-theory discussed in \S\ref{ss.chow-K} as the main examples.  The general construction is described in \cite{fm}; see also \cite{gk}.  Briefly, a homology theory $U_*$ is a functor from $\catC$ to groups, covariant for confined morphisms.  The associated operational bivariant theory $\op U$ is defined by taking operators $(c_g)\in \op U(f\colon X \to Y)$ to be collections of homomorphisms $c_g\colon U_*(Y')\to U_*(X')$, one for each independent square
\begin{diagram}
 X' & \rTo^{f'} & Y' \\
 \dTo & &  \dTo_g \\
 X & \rTo^f & Y,
\end{diagram}
subject to compatibility with pullback across independent squares and pushforward along confined morphisms.

This is usually refined by specifying a collection $\mathcal{Z}$ of distinguished operators, and passing to the smaller bivariant theory $\op U_\mathcal{Z}$ consisting of operators that commute with the Gysin maps determined by $\mathcal{Z}$.  The collection $\mathcal{Z}$ is part of the data of the bivariant theory.  For example, in operational Chow or $K$-theory, $\mathcal{Z}$ consists of the orientation classes $[f]$ associated to regular embeddings or flat morphisms, as described in \S\ref{ss.chow-K}.  When $\mathcal{Z}$ is clear from context, we omit the subscript, and write simply $\op U$.

We construct Grothendieck transformations using the following observation:

\begin{proposition}\label{p.transf}
Let $\catC$ and $\bar\catC$ be categories with homology theories $U_*$ and $\bar{U}_*$, respectively, with associated operational bivariant theories $\op U$ and $\op\bar{U}$.  Suppose $F\colon \catC \to \bar\catC$ is a functor preserving final objects, confined morphisms, and independent squares, with a left adjoint $L\colon \bar\catC \to \catC$, such that for all objects $\bar{X}$ of $\bar\catC$, the canonical map $\bar{X} \to FL(\bar{X})$ is an isomorphism.

Then any natural isomorphism $\tau\colon U_* \to \bar{U}_*\circ F$ extends canonically to a Grothendieck transformation $t\colon \op U \to \op\bar{U}$. 
Furthermore, if all operators in $\bar{\mathcal{Z}}$ are contained in the subgroups generated by $t(\mathcal{Z})$, then $t$ induces a Grothendieck transformation $t\colon \op U_{\mathcal{Z}} \to \op\bar{U}_{\bar{\mathcal{Z}}}$.
\end{proposition}

In the proposition and proof below, $\bar{X}$, etc., denotes an arbitrary object of $\bar\catC$, and we write $F(X)$, etc., for the images of objects under the functor $F$.

\begin{proof}
The transformation is constructed as follows.  Suppose we are given $c\in \op U(X \to Y)$ and a map $\bar{g}\colon \bar{Y}' \to F(Y)$.  Continuing our notation for fiber products, let $\bar{X}'= F(X) \times_{F(Y)} \bar{Y}'$ and $X' = X \times_Y L(\bar{Y}')$.  By the hypotheses on $F$ and $L$, there is a natural isomorphism $\bar{X}' \xrightarrow{\sim} F(X')$.

Now define $t(c)_{\bar{g}} \colon \bar{U}_*(\bar{Y}') \to \bar{U}_*(\bar{X}')$ as the composition
\[
   \bar{U}_*(\bar{Y}') = \bar{U}_*(FL(\bar{Y}')) \xrightarrow{\tau^{-1}} U_*(L(\bar{Y}')) \xrightarrow{c_g} U_*(X') \xrightarrow{\tau} \bar{U}_*(F(X')) =  \bar{U}_*(\bar{X}'),
\]
where $g\colon L(\bar{Y}') \to Y$ corresponds to $\bar{g}\colon \bar{Y}' \to F(Y)$ by the adjunction.  
The proof that this defines a Grothendieck transformation is a straightforward verification of the axioms.
\end{proof}

The prototypical example of a Grothendieck transformation and Riemann-Roch formula relates $K$-theory to Chow.  When $f$ is a proper smooth morphism, the class $u_f$ is given by the Todd class of the relative tangent bundle, $\td(T_f)$.  The transformation $t^.$ is the Chern character, and the commutativity of the first diagram is the Grothendieck-Riemann-Roch theorem,
\[
  \ch(f_*(\alpha)) = f_*(\ch(E)\cdot \td(T_f)).
\]
The commutativity of the second diagram is the Verdier-Riemann-Roch theorem; there is a unique functorial transformation $t_. = \tau \colon K_\circ(X) \to A_*(X)_\QQ$ that extends the Chern character for smooth varieties, and satisfies
\[
\tau(f^!(\beta)) = \td(T_f)\cdot f^!(\tau(\beta))
\]
for all $\beta\in K_\circ(Y)$, whenever $f\colon X \to Y$ is an lci morphism. These two theorems were refined in \cite{bfm}, and \cite{fg}, respectively, to include the case where $f$ is a proper lci morphism of possibly singular varieties.

\section{Operational Grothendieck-Verdier-Riemann-Roch} \label{s.GVRR}

The equivariant Riemann-Roch theorem of Edidin and Graham \cite{eg-rr} states that there are natural homomorphims
\[
  K^T_\circ(X) \to \hat{K}^T_\circ(X)_\QQ \xrightarrow{\tau} \aA^T_*(X)_\QQ,
\]
the second of which is an isomorphism.  Here $\aA^T_*(X)$ is the completion along the ideal of positive-degree elements in $A_T^*(\pt) = \Sym^*M$.  Combining with Proposition~\ref{p.transf} and Lemma~\ref{l.order}, we obtain a bivariant Riemann-Roch theorem.

\begin{theorem}\label{t.rr}
There are Grothendieck transformations
\[
  \opk_T^\circ(X\to Y) \to \opkK_T^\circ(X\to Y)_\QQ \xrightarrow{t} \aA_T^*(X\to Y)_\QQ,
\]
the second of which is an isomorphism.

These transformations are compatible with the change-of-groups homomorphisms constructed in Appendix~\ref{s.A.change-groups}.  If $T' \subseteq T$ is a subtorus, the diagram
\begin{diagram}
   \opk_T^\circ(X\to Y) &\rTo & \opkK_T^\circ(X\to Y)_\QQ & \rTo & \aA_T^*(X\to Y)_\QQ \\
     \dTo               &     & \dTo                &       & \dTo \\
   \opk_{T'}^\circ(X\to Y) & \rTo & \opkK_{T'}^\circ(X\to Y)_\QQ & \rTo & \aA_{T'}^*(X\to Y)_\QQ
\end{diagram}
commutes.
\end{theorem}

\begin{proof}
The transformation from $\opk_T^\circ$ to $(\opkK_T^\circ)_\QQ$ is completion and tensoring by $\QQ$, so there is nothing to prove.  To obtain the second transformation, we apply Proposition~\ref{p.transf}, taking $F$ to be the identity functor.  The only subtlety is in showing that $t$ takes the operations commuting with classes in $\mathcal{Z}$ (refined pullbacks for smooth morphisms and regular embeddings, in $K$-theory) to ones commuting with those in $\bar{\mathcal{Z}}$ (the same pullbacks in Chow theory).  (By Remark~\ref{r.weaker}, commutativity with flat pullback can be weakened to just smooth pullback without affecting the bivariant theories $A_T^*$ and $\opk_T^\circ$.)
Consider the diagram
\begin{diagram}
 X'' & \rTo & Y'' & \rTo  &  Z''  \\
\dTo & & \dTo_{h'}    &       & \dTo_{h} \\
 X' & \rTo & Y'    & \rTo  &  Z' \\
\dTo & & \dTo_g \\
X  & \rTo^{f}  & Y,
\end{diagram}
where $h$ is a smooth morphism or a regular embedding.  Let $\td=\td(T_h)$ be the equivariant Todd class of the virtual tangent bundle of $h$, and let $\alpha \in \aA_T^*(Y')_\QQ$ and $c \in \opkK_T^\circ( X \to Y )_\QQ$.  Using the equivariant Riemann-Roch isomorphism $\tau\colon (\kK^T_\circ)_\QQ \to (\aA^T_*)_\QQ$, we compute
\begin{align*}
  \tau ( c_{gh'} ( \tau^{-1}( h^! \alpha ) ) ) &= \tau ( c_{gh'} ( \tau^{-1}( \td\cdot \td^{-1}\cdot h^! \alpha ) ) ) \\
  &= \td^{-1} \cdot\tau ( c_{gh'} ( h^!( \tau^{-1}  \alpha  ) ) ) \\
  &= \td^{-1}\cdot \tau ( h^! ( c_g( \tau^{-1}  \alpha  ) ) ) \\
  &= \td^{-1}\cdot \td\cdot h^! ( \tau ( c_g( \tau^{-1}  \alpha  ) ) ) \\
  &= h^!  ( \tau ( c_g( \tau^{-1}  \alpha  ) ) ),
\end{align*}
as required.

For compatibility with change-of-groups, apply \cite[Proposition~3.2]{eg-rr}, observing that the tangent bundle of $T/T'$ is trivial, so its Todd class is $1$.
\end{proof}

\section{Adams-Riemann-Roch}\label{s.ARR}

We briefly recall that $K_T^\circ(X)$ is a $\lambda$-ring and hence carries \emph{Adams operations}.  These are ring endomorphisms $\psi^j$, indexed by positive integers $j$, and characterized by the properties:

\begin{enumerate}[(a)]

\medskip
\item For any line bundle $L$, $\psi^j[L] = [L^{\otimes j}]$, and

\medskip
\item For any morphism $f\colon X \rightarrow Y$, $f^* \circ \psi^j = \psi^j \circ f^*$. \label{item:compatible}

\medskip
\end{enumerate}
Adams operations do not commute with (derived) push forward under proper morphisms, but the failure to commute is quantified precisely by the equivariant Adams-Riemann-Roch theorem, at least when $f$ is a projective local complete intersection morphism and $X$ has the $T$-equivariant resolution property, as is the case when $X$ is smooth.  The role of the Todd class for the Adams-Riemann-Roch theorem is played by the equivariant Bott elements $\theta^j(T_f^\vee) \in K^\circ_T(X')$, where $T_f^\vee$ is the virtual cotangent bundle of the lci morphism $f$.  The Bott element $\theta^j$ is a homomorphism of (additive and multiplicative) monoids
\[
 \theta^j \colon (K_T^\circ(X)^+, + ) \to (K_T^\circ(X), \cdot),
\]
where $K_T^\circ(X)^+ \subseteq K_T^\circ(X)$ is the monoid of {\em positive elements}, generated---as a monoid---by classes of vector bundles.
It is characterized by the properties

\begin{enumerate}[(a)]

\medskip
\item For any equivariant line bundle $L$, $\theta^j(L) = 1 + L + \cdots + L^{j-1}$, and

\medskip
\item For any equivariant morphism $g\colon X'' \rightarrow X'$, $g^* \theta^j = \theta^j g^*$.

\medskip
\end{enumerate}
For example, $\theta^j(1)=j$, and more generally $\theta^j(n) =j^n$.  If $j$ is inverted in $K^\circ_T(X)$, then the Bott element $\theta^j$ extends to all of $K_T^\circ(X)$, and becomes a homomorphism from the additive to the multiplicative group of $\hat{K}_T^\circ(X)[j^{-1}]$.  That is, $\theta^j(c)$ is invertible in $\widehat K^\circ_T(X)[{j}^{-1}]$, for any $c \in K^\circ_T(X)$.

\begin{theorem}[{\cite[Theorem~4.5]{Kock98}}]\label{t.ARR-kock}
Let $X$ be a $T$-variety with the resolution property, and let $f\colon X' \rightarrow X$ be an equivariant projective lci morphism.  Then, for every class $c \in K^\circ_T(X')$,
\begin{equation}\label{e.ARR}
\psi^j f_* (c) = f_*( \theta^j(T_f^\vee)^{-1} \cdot \psi^j(c)),
\end{equation}
in $\widehat K^\circ_T(X)[{j}^{-1}]$.
\end{theorem}

We will define Adams operations in operational $K$-theory, and prove an operational bivariant generalization of this formula.  First, we must review the construction of the covariant Adams operations
\[
  \psi_j \colon K^T_\circ(X) \to \hat{K}^T_\circ(X)[j^{-1}].
\]
A (non-equivariant) version for quasi-projective schemes appears in \cite[\S7]{soule}.  We eliminate the quasi-projective hypotheses using Chow envelopes; see Remark~\ref{r.adams-env}.

For quasi-projective $X$, choose a closed embedding $\iota\colon X \hookrightarrow M$ in a smooth variety $M$.  By $K^\circ_T(M\text{ on }X)$, we mean the Grothendieck group of equivariant perfect complexes on $M$ which are exact on $M\setminus X$.  This is isomorphic to $\opk_T^\circ(X\hookrightarrow M)$, which in turn is identified with $K^T_\circ(X)$ via the Poincar\'e isomorphism.  We sometimes will denote this isomorphism by $\iota_*\colon K^T_\circ(X) \xrightarrow{\sim} K_T^\circ(M\text{ on }X)$.

Working with perfect complexes on $M$ has the advantage of coming with evident Adams operations: one defines endomorphisms $\psi^j$ of the $K_T^\circ(M)$-module $K_T^\circ(M\text{ on }X)$ by the same properties as the usual Adams operations.  To make this independent of the embedding, we must correct by the Bott element.  Here is the definition for quasi-projective $X$: the module homomorphism $\psi_j\colon K^T_\circ(X) \to \hat{K}^T_\circ(X)[j^{-1}]$ is defined by the formula
\[
  \psi_j(\alpha) := \theta^j(T_M^\vee)^{-1}\cdot \psi^j(\iota_*\alpha),
\]
where $T_M$ is the tangent bundle of $M$.

\begin{lemma}
The homomorphism $\psi_j$ is independent of the choice of embedding $X\hookrightarrow M$.  Furthermore, it commutes with proper pushforward: if $f\colon X \to Y$ is an equivariant proper morphism of quasi-projective schemes, then $f_*\psi_j(\alpha) = \psi_j(f_*\alpha)$ for all $\alpha\in K^T_\circ(X)$.
\end{lemma}

\begin{proof}
To see $\psi_j$ is independent of $M$, we apply the Adams-Riemann-Roch theorem for nonsingular quasi-projective varieties.   Given two embeddings $\iota\colon X \hookrightarrow M$ and $\iota'\colon X \hookrightarrow M'$, consider the product embedding $X\hookrightarrow M \times M'$, with projections $\pi$ and $\pi'$.  Let us write $\theta^j_M$ for $\theta^j(T_M^\vee)$, etc., and suppress notation for pullbacks, so for instance $\theta^j(T_\pi^\vee) = \theta^j_{M'}$.  Let us temporarily write $\psi_j^M(\alpha) = (\theta_M^j)^{-1}\cdot \psi^j(\iota_*\alpha)$ for the Adams operation with respect to the embedding in $M$, and similarly for $M'$ and $M\times M'$.

Using the projection $\pi\colon M \times M' \to M$ to compare embeddings, we have
\begin{align*}
  \psi_j^M(\alpha) &= (\theta_M^j)^{-1}\cdot \psi^j(\iota_*\alpha) \\
                   &= (\theta_{M'}^j)\cdot (\theta_{M\times M'}^j)^{-1} \cdot \psi^j(\pi_*(\iota\times\iota')_*\alpha) \\
                   &= \pi_*\big( (\theta_{M\times M'}^j)^{-1} \cdot \psi^j((\iota\times\iota')_*\alpha) \big) \qquad (\text{by \eqref{e.ARR}}) \\
                   &= \psi_j^{M\times M'}(\alpha),
\end{align*}
and similarly one sees $\psi_j^{M'}(\alpha) = \psi_j^{M\times M'}(\alpha)$.

Covariance for equivariant proper maps is similar.  Given such a map $f\colon X \to Y$ between quasi-projective varieties, one can factor it as in the following diagram:
\begin{diagram}
  X & \rInto & M \times Y & \rInto &  M \times M' \\
   & \rdTo_f &     \dTo   &       & \dTo \\
   &        &  Y        & \rInto   &   M'.
\end{diagram}
Here $M$ and $M'$ are smooth schemes into which $X$ and $Y$ embed, respectively.  Abusing notation slightly, we write
\[
 f_*\colon K_T^\circ( M\times M' \text{ on }X ) \to K_T^\circ( M'\text{ on }Y )
\]
for the pushforward homomorphism corresponding to $f_*\colon K^T_\circ(X) \to K^T_\circ(Y)$ under the canonical isomorphisms.  Computing as before, we have
\begin{align*}
 f_*\psi_j(\alpha) &= f_*\big( (\theta_{M\times M'}^j)^{-1}\cdot \psi^j(\iota_*\alpha) \big) \\
                  &= f_*\big( (\theta_{M}^j)^{-1}(\theta_{M'}^j)^{-1}\cdot \psi^j(\iota_*\alpha) \big)  \\
                  &= (\theta_{M'}^j)^{-1} \psi^j(f_*\alpha) \qquad (\text{by \eqref{e.ARR}}) \\
                  &= \psi_j(f_*\alpha),
\end{align*}
as claimed.
\end{proof}

\begin{remark}\label{r.adams-env}
To define covariant Adams operations for a general variety $X$, we choose an equivariant Chow envelope $X' \to X$, with $X'$ quasi-projective, and apply the descent sequence \eqref{e.k-descent}:
\begin{diagram}
  K^T_\circ(X'\times_X X') & \rTo & K^T_\circ(X') & \rTo & K^T_\circ(X) & \rTo & 0 \\
  \dTo_{\psi_j}     &             & \dTo_{\psi_j}  &     & \dDashto_{\psi_j} \\
  \hat{K}^T_\circ(X'\times_X X')[j^{-1}] & \rTo & \hat{K}^T_\circ(X')[j^{-1}] & \rTo & \hat{K}^T_\circ(X)[j^{-1}] & \rTo & 0   .
\end{diagram}
The two vertical arrows on the left are the Adams operations constructed above for quasi-projective schemes, and the corresponding square commutes thanks to covariance; this constructs the dashed arrow on the right.
\end{remark}

\begin{lemma}\label{l.adams-isom}
The Adams operations $\psi_j$ induce isomorphisms $\hat{K}^T_\circ(X)[j^{-1}] \xrightarrow{\sim} \hat{K}^T_\circ(X)[j^{-1}]$.
\end{lemma}

\begin{proof}
We start with the special case where $X$ is smooth and $T$ is trivial.  In this case, one sees that $\psi^j\colon K^\circ(X) \to K^\circ(X)$ becomes an isomorphism after inverting $j$ using the filtration 
by the submodules $F_\gamma^n \subset K^\circ(X)$ spanned by $\gamma$-operations of weight at least $n$.  A general fact about $\lambda$-rings is that $\psi^j$ preserves the $\gamma$-filtration, and acts on the factor $F_\gamma^n/F_\gamma^{n+1}$ as multiplication by $j^n$.  (See, e.g., \cite[\S III]{fulton-lang} for general facts about $\gamma$-operations and this filtration.)  Inverting $j$ therefore makes $\psi^j$ an automorphism of $K^\circ(X)[j^{-1}]$.  Since the Bott elements $\theta^j$ also become invertible, it follows that $\psi_j$ is an automorphism of $K_\circ(X)[j^{-1}] \isom K^\circ(X)[j^{-1}]$.

Still assuming $T$ is trivial, we now allow $X$ to be singular.  If $X$ is quasi-projective, embed it as $X\hookrightarrow M$.  Restricting the $\gamma$-filtration from $K^\circ(M)$ to $K^\circ(M\text{ on }X) \isom K_\circ(X)$, the above argument shows that $\psi_j$ becomes an isomorphism after inverting $j$.  
For general $X$, apply descent as in Remark~\ref{r.adams-env}.

Finally, the completed equivariant groups $\hat{K}^T_\circ(X)[j^{-1}]$ are a limit of non-equivariant groups $K_\circ(\mathbb{E}\times^T X)[j^{-1}]$, taken over finite-dimensional approximations $\mathbb{E} \to \mathbb{B}$ to the universal principal $T$-bundle \cite[\S2.1]{eg-rr}.  Since $\psi_j$ induces automorphisms on each term in the limit, it also induces an automorphism of $\hat{K}^T_\circ(X)[j^{-1}]$.
\end{proof}

\begin{theorem}\label{t.arr}
There are Grothendieck transformations
\[
  \opk_T^\circ(X\to Y) \xrightarrow{\psi^j} \opkK_T^\circ(X\to Y)[j^{-1}],
\]
that specialize to $\psi_j\colon \hat{K}^T_\circ(X) \to \hat{K}^T_\circ(X)[j^{-1}]$ when $Y$ is smooth.

These operations commute with the change-of-groups homomorphisms, and with the Grothendieck-Verdier-Riemann-Roch transformations of Theorem~\ref{t.rr}.
\end{theorem}

The statement that these generalized Adams operations commute with the Grothendieck-Verdier-Riemann-Roch transformation means that the diagram
\begin{diagram}
  \opk_T^\circ(X\to Y) &\rTo^{\psi^j}&  \opkK_T^\circ(X \to Y)[j^{-1}] \\
   \dTo^t    &    &  \dTo_t \\
   \aA_T^*(X \to Y)_\QQ &\rTo^{\psi^j_A} & \aA_T^*(X \to Y)_\QQ
\end{diagram}
commutes, where $\psi_A^j$ is defined to be multiplication by $j^k$ on $A_T^k(X\to Y)_\QQ$.

\begin{proof}
To construct the transformation, one proceeds exactly as for Theorem~\ref{t.rr}: taking $F$ to be the identity functor, we apply Proposition~\ref{p.transf} to the natural isomorphism $\psi_j \colon \hat{K}^T_\circ(-)[j^{-1}] \to \hat{K}^T_\circ(-)[j^{-1}]$.  Composing the resulting Grothen\-dieck transformation with the one given by inverting $j$ and completing produces the desired Adams operation.  This agrees with $\psi_j$ on $K^T_\circ(X) = \opk_T^\circ(X\to\pt)$ by construction, so it also agrees with $\psi_j$ for $K^T_\circ(X) = \opk_T^\circ(X\to Y)$ when $Y$ is smooth, using the Poincar\'e isomorphism.

Commutativity with the change-of-groups homomorphism is evident from the definition.  Commutativity with $t$ comes from the corresponding fact for the Chern character in the smooth case \cite[\S III]{fulton-lang}; the general case follows using embeddings of quasi-projective varieties and Chow descent.
\end{proof}

The Adams-Riemann-Roch formula from the Introduction is a consequence.

\begin{remark}
The Adams operations on the cohomology component $\opk_T^\circ(X)$ have the following simple and useful alternative construction.  Since $\opk_T^\circ$ is the right Kan extension of $K_T^\circ$ on smooth schemes, there is a natural isomorphism
\begin{equation}\label{e.opk-limit}
\opk^\circ_T(X) \cong \varprojlim_{g \colon X' \rightarrow X} K^\circ_T(X'),
\end{equation}
where the limit is taken over $T$-equivariant morphisms to $X$ from smooth $T$-varieties $X'$.  Hence we may define
\[
\psi^j\colon \opk_T^\circ(X) \to \opkK_T^\circ(X)[j^{-1}]
\]
as the limit of Adams operations on $K_T^\circ(X')$.  Similarly, for a projective equivariant lci morphism $f\colon X \to Y$, and any element $c \in \opk^\circ_T(X)$, the identity
\[
\psi^j f_* (c) = f_*( \theta^j(T_f^\vee)^{-1} \cdot \psi^j(c)),
\]
in $\opkK^\circ_T(Y)[j^{-1}]$ may be checked componentwise in $\hat{K}^\circ_T(Y')$, for each $Y' \to Y$ with $Y'$ smooth; in this context, the formula is that of Theorem~\ref{t.ARR-kock}.

Other natural and well-known properties of Adams operations that hold in the equivariant $K$-theory of smooth varieties carry over immediately, provided that they can be checked component by component in the inverse limit.  For instance, the subspace of $\opk^\circ_T(X)$ on which the Adams operation $\psi^j$ acts via multiplication by $j^n$ is independent of $j$, for any positive integer $n$, since the same is true in $K^\circ_T(X')$ for all smooth $X'$ mapping to $X$ \cite[Corollary~5.4]{Kock98}.

Similarly, when $X$ is a toric variety, the Adams operation $\psi^j$ on $K^\circ_T(X)$ agrees with pullback $\phi_j^*$, for the natural endomorphism $\phi_j \colon X \rightarrow X$ induced by multiplication by $j$ on the cocharacter lattice, whose restriction to the dense torus is given by $t \mapsto t^j$ \cite[Corollary~1]{Morelli93}.  Applying the Kimura exact sequence and equivariant resolution of singularities, it follows that the Adams operations on $\opk^\circ_T(X)$ agree with $\phi_j^*$, as well.
\end{remark}

\section{Localization theorems and Lefschetz-Riemann-Roch}\label{s.LRR}

Consider the categories $\catC = T\text{-}\mathbf{Sch}$ of $T$-schemes and equivariant morphisms, and $\bar{\catC} = \mathbf{Sch}$ of schemes with trivial $T$-action (and all morphisms), considered as a full subcategory of $\catC$.  Taking the fixed point scheme $F(X) = X^T$ defines a functor from $\catC$ to $\bar{\catC}$ preserving proper morphisms and fiber squares \cite[Proposition A.8.10]{cgp}; it is right adjoint to the embedding $\bar\catC \to \catC$.

Let $S\subseteq R(T)$ be the multiplicative set generated by $1-\ee^{-\lambda}$ for all $\lambda\in M$.  By \cite[Th\'eor\`eme~2.1]{thomason-lef}, the homomorphism
\begin{equation}
 S^{-1}\iota_*\colon S^{-1}K^T_\circ(X^T) \to S^{-1}K^T_\circ(X) \label{e.k-loc}
\end{equation}
is an isomorphism for any $T$-scheme $X$.

Similarly, let $\bar{S} \subseteq \Lambda_T = \Sym^*M$ be the multiplicative set generated by all $\lambda\in M$.  By \cite[\S2.3, Corollary~2]{brion-chow}, the homomorphism
\begin{equation}
  \bar{S}^{-1}\iota_*\colon \bar{S}^{-1} A^T_*(X^T) \to \bar{S}^{-1}A^T_*(X) \label{e.a-loc}
\end{equation}
is an isomorphism for any $T$-scheme $X$.

\begin{theorem}\label{t.local}
The fixed point functor $F(X) = X^T$ gives rise to Grothendieck transformations
\begin{align*}
  S^{-1}\opk_T^\circ(X\to Y) &\xrightarrow{\loc^K}  S^{-1} \opk_T^\circ(X^T \to Y^T)  
 \quad \text{and}\\
  \bar{S}^{-1} A_T^*(X \to Y) &\xrightarrow{\loc^A} \bar{S}^{-1} A_T^*(X^T \to Y^T) ,
\end{align*}
inducing isomorphisms of $S^{-1}R(T)$-modules and $\bar{S}^{-1}\Lambda_T$-modules, respectively.

These transformations commute with the equivariant Grothendieck-Verdier- and Adams-Riemann-Roch transformations: the diagrams
\begin{diagram}
  S^{-1}\opk_T^\circ(X\to Y) &\rTo^{\loc^K}&  S^{-1} \opk_T^\circ(X^T \to Y^T) \\
   \dTo^t    &    &  \dTo_t \\
  \bar{S}^{-1} \aA_T^*(X \to Y) &\rTo^{\loc^A} & \bar{S}^{-1} \aA_T^*(X^T \to Y^T)
\end{diagram}
and
\begin{diagram}
  S^{-1}\opk_T^\circ(X\to Y) &\rTo^{\loc^K}&  S^{-1} \opk_T^\circ(X^T \to Y^T) \\
   \dTo^{\psi^j}    &    &  \dTo_{\psi^j} \\
  \bar{S}^{-1} \opkK_T^\circ(X \to Y) &\rTo^{\loc^K} & \bar{S}^{-1} \opkK_T^\circ(X^T \to Y^T)
\end{diagram}
commute.
\end{theorem}

\begin{proof}
First, observe that if $X$ and $Y$ have trivial $T$-action, then
\[
  S^{-1}\opk_T^\circ(X \to Y) = S^{-1}R(T) \otimes_\ZZ \opk^\circ(X \to Y)
\]
canonically, by applying Lemma~\ref{l.order} to Kan extension along the inclusion of $(\mathbf{Sch})$ in $(T\text{-}\mathbf{Sch})$ as the subcategory of schemes with trivial action.  Letting $\bar{U}_*$ be the homology theory on $(\mathbf{Sch})$ given by $X \mapsto S^{-1}R(T)\otimes K_\circ(X)$, it follows that $S^{-1}\opk_T^\circ(X \to Y) = \mathrm{op}\bar{U}(X \to Y)$ for schemes with trivial $T$-action.

Since $X^T=F(X)$ has a trivial $T$-action, the target of $\loc^K$ may be identified with $\mathrm{op}\bar{U}(F(X) \to F(Y))$.  Using the inverse of the isomorphism \eqref{e.k-loc} as ``$\tau$'' in the statement of Proposition~\ref{p.transf}, we obtain the desired Grothendieck transformation.  The construction of $\loc^A$ is analogous, using the isomorphism \eqref{e.a-loc}.

Commutativity with the Riemann-Roch transformation follows from commutativity of the diagrams
\begin{diagram}
 S^{-1}K^T_\circ(X^T) & \rTo & S^{-1}K^T_\circ(X) \\
  \dTo  &  & \dTo \\
 S^{-1}\kK^T_\circ(X^T) & \rTo & S^{-1}\kK^T_\circ(X) \\
 \dTo  &   & \dTo \\
 \bar{S}^{-1}\aA^T_*(X^T) & \rTo & \bar{S}^{-1}\aA^T_*(X) ,\\
\end{diagram}
where the top square commutes by functoriality of completion, and the bottom square commutes by functoriality of the Riemann-Roch map (for proper pushforward).  The situation for Adams operations is similar.
\end{proof}

\begin{remark}
In general, the Grothendieck transformations $\loc^K$ and $\loc^A$ are distinct from the pullback maps $\iota^*$ induced by the inclusion $\iota\colon Y^T \to Y$; indeed, the latter is a homomorphism
\[
  \iota^* \colon \opk_T^\circ(X\xrightarrow{f} Y) \to \opk_T^\circ( f^{-1}Y^T \to Y^T ),
\]
but the inclusion $X^T \subseteq f^{-1}Y^T$ may be strict, and the pushforward along this inclusion need not be an isomorphism.  However, for morphisms $f$ such that $X^T = f^{-1}Y^T$, the homomorphism specified by $\loc^K$ agrees with $\iota^*$.  For instance, this holds when $f$ is an embedding.  In particular, taking $f$ to be the identity, the homomorphisms
\[
  S^{-1}\opk_T^\circ(X) \to S^{-1}\opk_T^\circ(X^T)
\]
induced by $\loc^K$ are identified with the pullback $\iota^*$.  The same holds for $\loc^A$.
\end{remark}

\section{Todd classes and equivariant multiplicities}

The formal similarity between Riemann-Roch and localization theorems suggests that the localization analogue of the Todd class should play a central role.  This analogue is the {\it equivariant multiplicity}.

For a proper flat map of $T$-schemes $f\colon X \to Y$ such that the induced map $f^T\colon X^T \to Y^T$ of fixed loci is also flat, we seek a class $\eem(f)\in S^{-1}\opk^\circ_T(X^T)$ fitting into commutative diagrams
\begin{equation}\label{e.sga6}
\begin{diagram}
 S^{-1}\opk_T^\circ(X)   &\rTo^{\sim} & S^{-1}\opk_T^\circ(X^T) \\
 \dTo^{f_!} &   &  \dTo_{f^T_!(\;\; \cdot \eem(f))} \\
 S^{-1}\opk_T^\circ(Y)   & \rTo^{\sim} & S^{-1}\opk_T^\circ(Y^T)
\end{diagram}
\end{equation}
and
\begin{equation}\label{e.verdier}
\begin{diagram}
 S^{-1}K^T_\circ(Y)   &\rTo^{\sim} & S^{-1}K^T_\circ(Y^T) \\
 \dTo^{f^!} &   &  \dTo_{\eem(f)\cdot (f^T)^!} \\
 S^{-1}K^T_\circ(X)   & \rTo^{\sim} & S^{-1}K^T_\circ(X^T).
\end{diagram}
\end{equation}
Or, more generally,
\begin{equation}\label{e.fm}
  \loc^K([f]) = \eem(f)\cdot [f^T]
\end{equation}
as bivariant classes in $S^{-1}\opk^\circ_T(X^T\to Y^T)$.

A unique such class exists when $f^T$ is smooth.  Indeed, product with $[f^T]$ induces a Poincar\'e isomorphism $\cdot[f^T]\colon \opk_T^\circ(X^T) \xrightarrow{\sim} \opk_T^\circ(X^T \to Y^T)$, so it can be inverted.

\begin{definition}
With notation and assumptions as above, when $f^T\colon X^T \to Y^T$ is smooth, the class
\[
  \eem^K(f) = \loc^K([f])\cdot [f^T]^{-1} \quad \text{ in } \quad S^{-1}\opk_T^\circ(X^T)
\]
is called the \define{total equivariant ($K$-theoretic) multiplicity of $f$}.  Restricting $\eem(f)$ to a connected component $P\subseteq X^T$ gives the equivariant multiplicity of $f$ along $P$,
\[
  \eem^K_P(f) \in S^{-1}\opk_T^\circ(P).
\]
The equivariant Chow multiplicities $\eem^A(f) \in \bar{S}^{-1}A_T^*(X^T)$ and $\eem^A_P(f) \in \bar{S}^{-1}A_T^*(P)$ are defined analogously.
\end{definition}

Recasting \eqref{e.sga6} with this definition gives an Atiyah-Bott pushforward formula.

\begin{proposition}
Suppose $f\colon X \to Y$ is proper and flat, and $f^T \colon X^T \to Y^T$ is smooth.  Let $Q\subseteq Y^T$ be a connected component.  For $\alpha\in \opk_T^\circ(X)$, we have
\begin{align}\label{e.abbv}
  (f_!\alpha)_Q &= \sum_{f(P)\subseteq Q} f^T_!(\alpha_P\cdot \eem^K_P(f)),
\end{align}
where $\beta_Q$ denotes restriction of a class $\beta$ to the connected component $Q$, and the sum on the RHS is over all components $P\subseteq X^T$ mapping into $Q$.
\end{proposition}

In general---when $f^T$ is flat but not smooth---we do not know when a class $\eem(f)$ exists.  However, smoothness of the map on fixed loci is automatic in good situations, e.g., when $X^T$ and $Y^T$ are finite and reduced.

Equivariant multiplicities for the map $X \to \pt$ will be denoted $\eem^K(X)$.  Suppose $X^T$ is finite and \emph{nondegenerate}, meaning that the weights $\lambda_1,\ldots,\lambda_n$ of the $T$-action on the Zariski tangent space $T_pX$ are all nonzero, for $p \in X^T$.  This implies that the scheme-theoretic fixed locus is reduced \cite[Proposition~A.8.10(2)]{cgp}, and hence $f^T\colon X^T \to \pt$ is smooth.

\begin{proposition}\label{p.nondeg-cone}
Suppose $p$ is a nondegenerate fixed point of $X$, and let $C$ be the tangent cone $C_pX \subseteq T_pX$ at $p$.  Then
\[
  \eem^K_p(X) = \frac{[\OO_{C}]}{(1-\ee^{-\lambda_1})\cdots (1-\ee^{-\lambda_n})} \quad \text{ and } \quad  \eem^A_p(X) = \frac{[C]}{\lambda_1 \cdots \lambda_n}
\]
in $S^{-1}R(T)$ and $\bar{S}^{-1}\Lambda_T$, respectively.  In particular, if $p\in X$ is nonsingular, 
\[
  \eem^K_p(X) = \frac{1}{(1-\ee^{-\lambda_1})\cdots (1-\ee^{-\lambda_n})} \quad \text{ and } \quad  \eem^A_p(X) = \frac{1}{\lambda_1 \cdots \lambda_n}.
\]
\end{proposition}

\noindent
The proposition justifies our terminology, because it implies the Chow multiplicity $\eem^A_p(X)$ agrees with the Brion-Rossmann equivariant multiplicity \cite{brion-chow,rossmann}.

\begin{proof}
From \eqref{e.verdier}, equivariant multiplicities have the characterizing property
\[
  [\OO_X] = \sum_{p\in X^T} \eem^K_p(X) \cdot [\OO_p]
\]
and
\[
  [X] = \sum_{p\in X^T} \eem^A_p(X) \cdot [p],
\]
under identifications $S^{-1}K^T_\circ(X) = S^{-1}K^T_\circ(X^T)$ and $\bar{S}^{-1}A^T_*(X) = \bar{S}^{-1}A^T_*(X^T)$.  Under deformation to the tangent cone at $p$, these equalities become
\[
  [\OO_C] = \eem^K_p(X) \cdot [\OO_p]
\]
and
\[
  [C] = \eem^A_p(X)\cdot [p]
\]
in $K^T_\circ(T_pX)=R(T)$ and $A^T_*(T_pX)=\Lambda_T$.  Since $[\OO_p] = (1-\ee^{-\lambda_1})\cdots (1-\ee^{-\lambda_n})$ in $K^T_\circ(T_pX)$ and $[p]=\lambda_1\cdots\lambda_n$ in $A^T_*(T_pX)$, the proposition follows.
\end{proof}

The formula for the $K$-theoretic multiplicity in the proposition gives $\eem^K_p(X)$ as a multi-graded Hilbert series:
\[
  \eem^K_p(X) = \sum_{\lambda\in M} (\dim_k \OO_{C,\lambda} )\cdot \ee^\lambda,
\]
where $\OO_{C,\lambda}$ is the $\lambda$-isotypic component of the rational $T$-module $\OO_C$ (cf.~\cite{rossmann}).

Built into our definition of equivariant multiplicity is another way of computing it, via resolutions.  Suppose $f\colon X \to Y$ is given, with both $X^T$ and $Y^T$ finite and nondegenerate.  Then if $f_*[\OO_X]=[\OO_Y]$, as is the case when $Y$ has rational singularities and $X\to Y$ is a desingularization, we have
\[
  \eem^K_q(Y) = \sum_{p\in (f^{-1}(q))^T} \eem^K_p(X).
\]
This often gives an effective way to compute $\eem_q(Y)$.

A fixed point $p$ is \emph{attractive} if all weights $\lambda_1,\ldots,\lambda_n$ lie in an open half-space.  

\begin{lemma}\label{nonzero.kmult}
If $p\in X^T$ is attractive then $\eem^K_p(X)$ is nonzero in $S^{-1}R(T)$.
\end{lemma}

The proof is similar to \cite[\S4.4]{brion-chow}, which gives the corresponding statement for Chow multiplicities $\eem^A_p(X)$. The $K$-theory version also follows from the Chow version; by Proposition~\ref{p.nondeg-cone}, the numerator and denominator of $\eem^A_p(X)$ are the leading terms of the numerator and denominator of $\eem^K_p(X)$, respectively.

\begin{lemma}\label{inj.pd}
Let $X$ be a complete $T$-scheme such that all fixed points in $X$ are nondegenerate.  If all equivariant multiplicities are non-zero, then the canonical map $\opk_T^\circ(X)\to K_\circ^T(X)$, sending $c\mapsto c(\OO_X)$, is injective.
\end{lemma}

The proof is similar to that of \cite[Theorem 4.1]{go.rm}, which gives the analogous result for Chow; we omit the details.  Using Lemma~\ref{nonzero.kmult}, the hypothesis of Lemma~\ref{inj.pd} is satisfied whenever all fixed points are attractive.

\begin{example} \label{examples.attractive.rem}
Lemma~\ref{inj.pd} applies to:
(i) projective nonsingular $T$-varieties with isolated fixed points (by Proposition~\ref{p.nondeg-cone}); 
(ii)
Schubert varieties and complete toric varieties, as they have only attractive fixed points; 
(iii) projective $G\times G$-equivariant embeddings of a
connected reductive group $G$, 
as they have only finitely many $T\times T$-fixed points, all of which are attractive.
\end{example}


\begin{remark}\label{r.todd}
The formal analogy between Riemann-Roch and localization theorems was observed by Baum-Fulton-Quart \cite{bfq}.  In fact, the relationship between Todd classes and equivariant multiplicities can be made more precise, as follows.  Assume $f\colon X \to Y$ is proper and lci, and $f^T\colon X^T \to Y^T$ is smooth.  From Theorem~\ref{t.local} and the Riemann-Roch formulas, we have
\begin{align*}
   t(\eem^K(f))\cdot t([f^T]) &= t(\loc^K([f]))\\
                  & = \loc^A( t([f]) ) \\
                  & = \loc^A( \td(T_f) )\cdot \eem^A(f)\cdot[f^T].
\end{align*}

\noindent
In particular, when $X^T$ is finite and nondegenerate, and $Y=\pt$,
\[
  \td(X)|_p = \frac{\ch(\eem^K_p(X))}{\eem^A_p(X)}.
\]
If $X$ is nonsingular at $p$, with tangent weights $\lambda_1,\ldots,\lambda_n$, this recovers a familiar formula for the Todd class:
\[
  \td(X)|_p = \prod_{i=1}^n\frac{\lambda_i}{1-\ee^{-\lambda_i}}.
\]
An analogous calculation, applied to Adams-Riemann-Roch, produces similar formulas for the localization of equivariant Bott elements.
\end{remark}

\begin{remark}
Suppose $f\colon X \hookrightarrow Y$ and $f^T\colon X^T \hookrightarrow Y^T$ are both regular embeddings.  The {\it excess normal bundle} for the diagram
\begin{diagram}
 X^T & \rInto & Y^T  \\
 \dInto & & \dInto \\
 X & \rInto & Y
\end{diagram}
is $E = (N_{X/Y}|_X)/(N_{X^T/Y^T})$.  In this situation, the class $\eem(f)$ satisfying \eqref{e.fm} is $\lambda_{-1}(E^*)$, where for any (equivariant) vector bundle $V$, the class $\lambda_{-1}(V)$ is defined to be $\sum (-1)^i [\exterior^i V]$.  The analogous class in bivariant Chow theory is $c^T_e(E)$, where $e$ is the rank of $E$.  (This is a restatement of the excess intersection formula.  For Chow groups, it is \cite[Proposition~17.4.1]{fulton-it}.  The proof is similar in K-theory; see, e.g., \cite[Theorem~3.8]{Kock98}.)
\end{remark}

\begin{remark}
The interaction between localization and Grothendieck-Riemann-Roch can be viewed geometrically as follows.  Using coefficients in the ground field, which we denote by $\CC$, we have $\Spec( R(T)\otimes \CC ) = T$ and $\Spec( \Lambda \otimes \CC ) = \mathfrak{t}$.  When $X=\pt$, the equivariant Chern character corresponds to the identification of a formal neighborhood of $0\in\mathfrak{t}$ with one of $1\in T$.

Now suppose $X$ has finitely many nondegenerate fixed points, and finitely many one-dimensional orbits, so it is a {\em $T$-skeletal variety} in the terminology of \cite{g.loc}.  The GKM-type descriptions of $\opk^\circ_T(X)$ (see \cite[Theorem~5.4]{g.loc} shows that $\Spec(\opk_T^\circ(X)_\CC)$ consists of copies of $T$, one for each fixed point, glued together along subtori.  Similarly, $\Spec( A_T^*(X)_\CC )$ is obtained by glueing copies of $\mathfrak{t}$ along subspaces.  There are structure maps $\Spec(\opk_T^\circ(X)_\CC) \to T$ and $\Spec( A_T^*(X)_\CC ) \to \mathfrak{t}$, and the  equivariant Chern character gives an isomorphism between fibers of these maps over formal neighborhoods of $1$ and $0$.  Equivariant multiplicities are rational functions on these spaces, regular away from the gluing loci.

A similar picture for topological $K$-theory and singular cohomology was described by Knutson and Rosu \cite{knutson-rosu}.
\end{remark}

\section{Toric varieties} \label{sec:toric}

Let $N=\Hom(M,\ZZ)$, and let $\Delta$ be a fan in $N_\RR$, i.e., a collection of cones $\sigma$ fitting together along common faces.  This data determines a toric variety $X(\Delta)$, equipped with an action of $T$.  (See, e.g., \cite{fulton-tv} for details on toric varieties.)

We now use operational Riemann-Roch to give examples of projective toric varieties $X$ such that the forgetful map $K_T^\circ(X) \to K^\circ(X)$ is not surjective.  

\begin{proposition}\label{p.not-surj}
Let $X=X(\Delta)$, where $\Delta$ is the fan over the faces of the cube with vertices at $\{ (\pm 1,\pm 1,\pm 1) \}$.  Then $K_T^\circ(X) \to K^\circ(X)$ is not surjective.
\end{proposition}

\begin{proof}
By \cite[Example~4.2]{kp}, the homomorphism $A_T^*(X)_\QQ \to A^*(X)_\QQ$ is not surjective, and therefore neither is the induced homomorphism $\alpha\colon \hat{A}_T^*(X)_\QQ \to A^*(X)_\QQ$.  Consider the diagram
\begin{diagram}
 K_T^\circ(X)_\QQ & \rTo & \opk_T^\circ(X)_\QQ & \rTo & \opkK_T^\circ(X)_\QQ & \rTo^\sim & \aA_T^*(X)_\QQ \\
 \dTo^\gamma   &    &  \dTo    &    &  \dTo   &   & \dTo_\alpha \\
 K^\circ(X)_\QQ  & \rOnto^\beta & \opk^\circ(X)_\QQ & = & \opk^\circ(X)_\QQ & \rTo^\sim & A^*(X)_\QQ.
\end{diagram}
By \cite[Theorem~1.4]{ap}, the homomorphism $\beta$ is surjective.  A diagram chase shows that $\gamma$ cannot be surjective.
\end{proof}

The same statement holds, with the same proof, for the other examples shown in \cite{kp} to have a non-surjective map $A_T^*(X)_\QQ \to A^*(X)_\QQ$.
\begin{question}
Can one find examples where $A_T^*(X)_\QQ \to A^*(X)_\QQ$ is surjective, but $K_T^\circ(X) \to K^\circ(X)$ is not?
\end{question}

Given a basis for $K^T_\circ(X)$, the dual basis for $\opk_T^\circ(X)=\Hom(K^T_\circ(X),R(T))$ can be computed using equivariant multiplicities, which are easy to calculate on a toric variety.  We illustrate this for a weighted projective plane.

\begin{example}
Let $N=\ZZ^2$, with basis $\{e_1,e_2\}$, and with dual basis $\{u_1,u_2\}$ for $M$.  Let $\Delta$ be the fan with rays spanned by $e_1$, $e_2$, and $-e_1-2e_2$; the corresponding toric variety $X=X(\Delta)$ is isomorphic to $\PP(1,1,2)$.  Let $D$ be the toric divisor corresponding to the ray spanned by $-e_1-2e_2$, and $p$ the fixed point corresponding to the maximal cone generated by $e_1$ and $-e_1-2e_2$.

Figure~\ref{f.emP112} shows the equivariant multiplicities for $X$, $D$, and $p$, arranged on the fan to show their restrictions to fixed points.  For the two smooth maximal cones, the multiplicities are computed by Proposition~\ref{p.nondeg-cone}; the singular cone (corresponding to $p$) can be resolved by adding a ray through $-e_2$.

\begin{figure}[h!]
\begin{pspicture}(-65,-50)(70,50)

\psline{->}(0,0)(35,0)
\psline{->}(0,0)(0,35)
\psline{->}(0,0)(-16,-32)

\rput[l](10,17){\small $\displaystyle{\frac{1}{(1-\ee^{u_1})(1-\ee^{u_2})}}$}
\rput[l](6,-16){\small $\displaystyle{\frac{1+\ee^{u_1-u_2}}{(1-\ee^{2u_1-u_2})(1-\ee^{-u_2})}}$}
\rput[r](-8,3){\small $\displaystyle{\frac{1}{(1-\ee^{-u_1})(1-\ee^{-2u_1+u_2})}}$}

\rput(0,-44){$\eem^K(X)$}

\end{pspicture}

\begin{pspicture}(-85,-50)(110,50)

\psline{->}(0,0)(35,0)
\psline{->}(0,0)(0,35)
\psline{->}(0,0)(-16,-32)

\rput[l](10,14){\small $0$}
\rput[l](6,-16){\small $\displaystyle{\frac{1}{1-\ee^{2u_1-u_2}}}$}
\rput[r](-6,3){\small $\displaystyle{\frac{1}{1-\ee^{-2u_1+u_2}}}$}

\rput(0,-44){$\eem^K(D)$}

\end{pspicture}
\begin{pspicture}(-75,-50)(80,50)

\psline{->}(0,0)(35,0)
\psline{->}(0,0)(0,35)
\psline{->}(0,0)(-16,-32)

\rput[l](10,14){\small $0$}
\rput[l](6,-14){\small $1$}
\rput[r](-8,0){\small $0$}

\rput(0,-42){$\eem^K(p)$}

\end{pspicture}

\caption{Equivariant multiplicities for $\PP(1,1,2)$ \label{f.emP112}}
\end{figure}
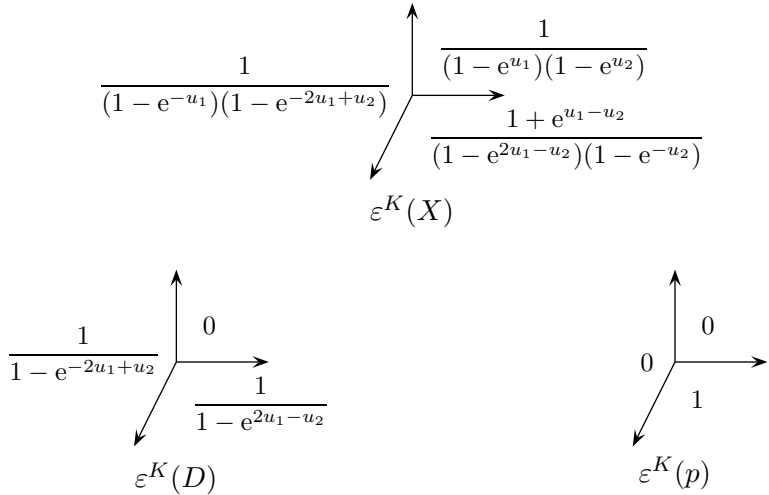

\noindent The classes $[\OO_X]$, $[\OO_D]$, and $[\OO_p]$ form an $R(T)$-linear basis for $K^T_\circ(X)$.  The dual basis for $\opk_T^\circ(X)$ was computed in \cite[Example~1.7]{ap}.  
The canonical map $\opk_T^\circ(X) \to K^T_\circ(X)$, sending $c\mapsto c(\OO_X)$, is then given by
\begin{align*}
[\OO_X]^\vee &\mapsto (1-\ee^{u_1})(1-\ee^{u_2})[\OO_X] + (\ee^{u_1}-\ee^{u_1+u_2})[\OO_D] + \ee^{u_2}[\OO_p] \,;\\
[\OO_D]^\vee &\mapsto (\ee^{u_1}-\ee^{u_1+u_2})[\OO_X] + (\ee^{-u_1+u_2}+\ee^{u_1+u_2}+\ee^{u_2}-\ee^{u_1})[\OO_D]\\
&\qquad  - ( \ee^{u_2}+ \ee^{-u_1+u_2})[\OO_p] \,;\\
[\OO_p]^\vee &\mapsto \ee^{u_2}[\OO_X] - (\ee^{u_2}+\ee^{-u_1+u_2})[\OO_D] + \ee^{-u_1+u_2}[\OO_p] .
\end{align*}
The resulting $3\times 3$ matrix has determinant $\ee^{-u_1+2u_2}+\ee^{u_2}$, which is not a unit in $R(T)$, and the map $\opk_T^\circ(X) \to K^T_\circ(X)$ is injective, but not surjective.
\end{example}

\begin{remark}
When $X$ is an affine toric variety, then it is easy to see $\opk_T^\circ(X) \isom R(T)$ and $A_T^*(X) \isom \Lambda$, for example by using the descriptions of these rings as piecewise exponentials and polynomials, respectively \cite{ap,payne-chow}.  (In fact, this is true more generally when $X$ is a $T$-skeletal variety with a single fixed point, see \cite{g.loc}.)  For non-equivariant groups, Edidin and Richey have recently shown that $\opk^\circ(X) \isom \ZZ$ and $A^*(X) \isom \ZZ$ \cite{edidin-richey,edidin-richey2}.  The relationship between the equivariant and non-equivariant groups is subtle.  On the other hand, one can use our Riemann-Roch theorems (together with the facts that $\opk^\circ(X)$ and $A^*(X)$ are torsion-free) to deduce the Chow statement from the $K$-theory one, or vice-versa.
\end{remark}

\section{Spherical varieties} \label{ss:spherical}

Let $G$ be a connected reductive linear algebraic group with Borel subgroup $B$ and maximal torus $T\subset B$.  A {\em spherical} variety is a $G$-variety with a dense $B$-orbit.  In other sources, spherical varieties are assumed to be normal, but here this condition is not needed and we do not assume it.  If $X$ is a spherical variety, then it has finitely many $B$-orbits, and thus also a finite number of $G$-orbits, each of which is also spherical. Moreover, since every spherical homogeneous space has finitely many $T$-fixed points, it follows that $X^T$ is finite.  Examples of spherical varieties include toric varieties, flag varieties, symmetric spaces, and $G\times G$-equivariant embeddings of $G$.  See \cite[\S5]{ti:sph} for references and further details.

In this section, we describe the equivariant operational $K$-theory
of a possibly singular complete spherical variety
using the following localization theorem.

\begin{theorem}[\cite{g.loc}]\label{loc.thm.cs}
Let $X$ be a $T$-scheme. If the action of $T$ has enough limits (e.g. if $X$ is complete),
then the restriction homomorphism $\opk_T^\circ(X)\to \opk_T^\circ(X^T)$ is injective,
and its image is the intersection of the images of the restriction homomorphisms
$
\opk_T^\circ(X^H)\to \opk_T^\circ(X^T),
$
where $H$ runs over all subtori of codimension one in $T$. \qed
\end{theorem}

When $X$ is singular, the fixed locus $X^H$ may be complicated: its irreducible components $Y_i$ may be singular, and they may intersect along subvarieties of positive dimension.  In this context, the restriction map $\opk_T^\circ(X^H) \to \bigoplus_{i}\opk_T^\circ(Y_i)$ is typically not an isomorphism.  The following lemma gives a method for overcoming this difficulty; it is proved in \cite[Remark~3.10]{g.loc}.

\begin{lemma}\label{l.irr}
Let $Y$ be a complete $T$-scheme with finitely many fixed points, let $Y_1,\ldots,Y_n$ be its irreducible components, and write $Y_{ij} = Y_i \cap Y_j$.  We identify elements of $\opk_T^\circ(Y^T)$ with functions $Y^T \to R(T)$, written $f\mapsto f_x$ (and similarly for $Y_i^T$).  In the diagram
\begin{diagram}
 \opk_T^\circ(Y) & \rTo & \bigoplus_i \opk_T^\circ(Y_i)  \\
\dTo_{\iota_Y^*} & & \dTo_{\oplus \iota_{Y_i}^*} \\
\opk_T^\circ(Y^T) & \rTo^{p}& \bigoplus_i \opk_T^\circ(Y^T_i) ,
\end{diagram}
all arrows are injective, and we have
\[
{\rm Im}(p\circ \iota^*_Y) = {\rm Im}(\oplus \iota_{Y_i}^*) \cap \{ (f^{(i)})_{i=1}^n \,|\, f^{(i)}_x=f^{(j)}_x \text{ for all } x\in Y_{ij}^T \}. 
\]
\end{lemma}

Applying Lemma~\ref{l.irr} to $Y=X^H$, we can identify the image of $\opk^\circ_T(X^H)$ in $\opk^\circ_T(X^T)$ by computing $\opk^\circ_T(Y_i)$ separately for each irreducible component $Y_i$, and identifying the conditions imposed on the restrictions to the finitely many $T$-fixed points.

\smallskip

For the rest of this section, $X$ is a complete spherical $G$-variety, and $H\subset T$ is a subtorus of codimension one.  Our goal is to compute $\opk_T^\circ(X^H)$, and we begin by studying the possibilities for the irreducible components of $X^H$.

A subtorus $H\subset T$ is {\em regular} if its centralizer $C_G(H)$ is equal to $T$.  In this case, $\dim(X^H)\leq 1$.  Let $Y$ be an irreducible component of $X^H$, so the torus $T$ acts on $Y$.  If $Y$ is a single point, or a curve with unique $T$-fixed point, then $\opk_T^\circ(Y)\isom R(T)$.  Otherwise, $T$ acts on the curve $Y$ via a character $\chi$, fixing two points, so $Y^T = \{x,y\}$, and we have
\[
\opk_T^\circ(Y) \isom \{ (f_x,f_y) \,|\, f_x - f_y \equiv 0 \mod (1-\ee^{-\chi}) \} \subseteq R(T)^{\oplus 2}.
\]
One can see this from the integration formula: we must have $\eem_x\cdot f_x + \eem_y \cdot f_y \in R(T)$, and clearing denominators in the requirement
\begin{align}\label{e.divisible}
 \frac{f_x}{1-\ee^{-\chi}} + \frac{f_y}{1-\ee^{\chi}} \in R(T)
\end{align}
leads to the asserted divisibility condition.  
(See \cite[Proposition 5.2]{g.loc}.)  This settles the case of regular subtori.

If the codimension-one subtorus $H$ is not regular, then it is {\em singular}.  A subtorus of codimension one is singular if and only if it is the identity component of the kernel of some positive root.  In this case, $C_G(H) \subseteq G$ is generated by $H$ together with a subgroup isomorphic to $SL_2$ or $PGL_2$.  
In particular, there is a nontrivial homomorphism $SL_2 \to C_G(H) \subseteq G$.  
By \cite[Proposition 7.1]{brion-chow}, each irreducible component of $X^{H}$ is spherical with respect to this $SL_2$ action, and $\dim(X^{H}) \leq 2$.

Analyzing the case of a singular codimension-one subtorus $H$ will take up most of the rest of this section.  We set the following notation.

\begin{notation}\label{n.fixed}
Let $H\subset T$ be a singular subtorus of codimension one, and let $\phi\colon G'=SL_2 \to C_G(H) \subseteq G$ be the corresponding homomorphism.  Let $B'=\phi^{-1}B \subset G'$, a Borel subgroup which may be identified with upper-triangular matrices in $SL_2$.

Let $D' = \phi^{-1}T \subset G'$, maximal torus which may be identified with diagonal matrices in $SL_2$.  We further identify $D'$ with $\mathbb{G}_m$ via $\zeta\mapsto {\rm diag}(\zeta^{-1},\zeta)$.  Let $D=\phi(D') \subseteq T$, a one-dimensional subgroup such that $T\isom D \times H$.

Finally, let $Y$ be an irreducible component of $X^H$, and let $\tilde{Y}$ be its normalization.  We consider both $Y$ and $\tilde{Y}$ as spherical $G'$-varieties via $\phi\colon G' \to G$. 
\end{notation}

To describe the geometry of the varieties $Y$ and $\tilde{Y}$, we use the classification of normal complete spherical varieties from \cite{a-sl2} (see also \cite[Example 2.17]{ab}).  By \cite{a-sl2}, the normal $G'$-variety $\tilde Y$ is equivariantly isomorphic to one of the following:

\medskip

\begin{enumerate}[itemsep=5pt]
\item A single point.

\item A projective line $\mathbb{P}^1=G'/B'$.

\item A projective plane $\mathbb{P}(V)$, on which $G'=SL_2$ acts by the projectivization of its linear action on $V=\Sym^2\CC^2$ (quadratic forms in two variables) with two orbits, the conic of degenerate forms and its complement, which is isomorphic to $G'/N_{G'}(D')$.

\item A product of two projective lines $\mathbb{P}^1\times \mathbb{P}^1$, on which $G'$ acts diagonally with two orbits, the diagonal and its complement, which is a dense orbit isomorphic to $G'/D'$.

\item A Hirzebruch surface $\mathbb{F}_n=\mathbb{P}(\mathcal{O}_{\mathbb{P}^1}\oplus \mathcal{O}_{\mathbb{P}^1}(n))$, $n\geq 1$,
on which $G'$ acts via its natural actions on $\mathbb{P}^1$ and the
linearized sheaf $\mathcal{O}_{\mathbb{P}^1}(n)$, with three orbits.  The dense orbit has isotropy group $U_n$, the semidirect product of a one-dimensional unipotent subgroup $U\subset B'$ with the subgroup of $n$-th roots of unity in $D'$, and the complement of this orbit consists of two closed orbits $C_+$ and $C_-$, which are sections of the fibration $\mathbb{F}_n\to \mathbb{P}^1$ with self-intersection $n$ and $-n$, respectively.

\item \label{Pn} A normal projective surface $P_n$ obtained from $\mathbb{F}_n$ by contracting the negative section $C_-$. In this case, $\tilde Y$  has three $G'$-orbits: the dense orbit with isotropy group $U_n$, the image of the positive section $C_+$, and a fixed point (the image of the contracted curve $C_-$).  For $n=1$, this case includes $P_1 \isom \PP^2$, a compactification of $SL_2$ acting on $\AA^2$ by the standard representation.  
\end{enumerate}

\medskip

Our first goal is to reduce to the case where $Y$ is normal, so that we can use the above classification.

\begin{lemma}\label{sl2.envelope}
Every $G'$-orbit in $Y$ is the isomorphic image of a $G'$-orbit in $\tilde{Y}$. In particular, the normalization $\pi\colon \tilde{Y}\to Y$ is a $G'$-equivariant envelope.
\end{lemma}

\begin{proof}
 Let $\mathcal{O}=G'\cdot x$ be an orbit in $Y$.  If $\mathcal{O}$ is open, then $\pi^{-1}(\mathcal{O})$ maps isomorphically to $\mathcal{O}$.  Suppose $\mathcal{O}$ is not open.  Then either $\mathcal{O}\simeq G'/B'$ or $\mathcal{O}$ is a $G'$-fixed point. In either case, the isotropy group $G'_x$ is connected, and hence acts trivially on $\pi^{-1}(x)$. Then, for any $y \in \pi^{-1}(x)$, $G'\cdot y$ maps isomorphically to $G'\cdot x$.
\end{proof}

\begin{corollary}\label{sl2.bijective.env}
The normalization $\pi \colon \widetilde Y \to Y$
is bijective unless $Y$ is a surface with a double curve obtained by identifying $C_+$ and $C_-$ in $\mathbb{F}_n$.
\end{corollary}

\noindent Such surfaces are complete and algebraic, but not projective.  See, e.g., \cite{kodaira}.  In particular, if $X$ is projective then $\pi$ is bijective for all $H$ and all $Y$.

\begin{proof}
By Lemma \ref{sl2.envelope}, every $G'$-orbit in $Y$ is the isomorphic image of an orbit in $\tilde{Y}$.  
Hence $Y$ has at most three $G'$-orbits.  
Let $y\in Y$.  If $y$ is in the open orbit, then $|\pi^{-1}(y)|=1$.  Otherwise, $y$ is in a closed orbit, and its stabilizer is either $G'$ or $B'$.  If $y$ is a $G'$-fixed point, then each point in $\pi^{-1}(y)$ is fixed. Since $\tilde{Y}$ has at most one $G'$-fixed point, we conclude that $|\pi^{-1}(y)|=1$.  Otherwise, the orbit of each $z\in \pi^{-1}(y)$ is a $G'$-curve in $\tilde{Y}$ mapping isomorphically to $\mathcal{O}_y$.

Consequently, $\pi$ is a bijection unless it identifies two $G'$-stable curves in $\tilde Y$.  From the classification above, we see that the only way this can happen is if $\tilde Y \cong \mathbb{F}_n$ and $\pi$ identifies the curves $C_+$ and $C_-$. 
It is worth noting that this gluing, being $G'$-equivariant, is uniquely determined. 
Indeed, to glue $C_+$ and $C_-$ so that the quotient inherits a $G'$-action, 
we should use a $G'$-equivariant isomorphism $C_+\to C_-$. 
The Borel subgroup $B'$ also acts on both curves, with unique fixed points $p_+\in C_+$ and $p_-\in C_-$. Thus 
an equivariant isomorphism must send $p_+$ to $p_-$. 
Since $C_+$ and $C_-$ are homogeneous for $G'$, this determines the map.   
\end{proof}

The previous corollary together with the Kimura sequence (eq.~\eqref{e.kimura2} of \S\ref{ss.chow-K}) implies the following: 

\begin{corollary}\label{c.normalization}
The normalization map $\pi\colon \tilde{Y} \to Y$ induces an isomorphism
$
\opk_{D'}^\circ(Y)\xrightarrow{\sim} \opk_{D'}^\circ(\tilde{Y}),
$
unless $Y$ is a surface with a double curve obtained by identifying $C_+$ and $C_-$ in $\mathbb{F}_n$. \qed  
\end{corollary}

\medskip
Since $T\isom D\times H$, it follows from \cite[Corollary 5.6]{ap} that
\[
  \opk_T^\circ(X^H) \isom \opk_{D}^\circ(X^H) \otimes R(H).
\]
Our analysis therefore reduces to computing $\opk_{D}^\circ(Y)$ in all 
cases listed above.  In each case, $Y$ has finitely many $D$-fixed points, so we will compute $\opk_{D}^\circ(Y)$ as a subring of $\opk_{D}^\circ(Y^{D})$, which is a direct sum of finitely many copies of $R(D)\isom R(\mathbb{G}_m)$.

Moreover, the homomorphism $D' \to D$ is either an isomorphism or a double cover, so the corresponding homomorphism $R(D) \to R(D')$ is either an isomorphism or an injection which may be identified with the inclusion $\ZZ[e^{\pm 2t}] \hookrightarrow \ZZ[e^{\pm t}]$.  
In view of Lemma~\ref{sl2.envelope} and its corollaries, then, it suffices to describe $\opk_{D'}^\circ(Y)$, where $Y$ is one of the six normal $G'$-varieties listed above, or the surface with a double curve obtained by identifying $C_+$ and $C_-$ in $\mathbb{F}_n$. In fact, if $\chi$ is a root of $G$, then the homomorphism $R(D) \to R(D')$ maps $e^\chi$  
to $e^{2t}$. When $D'\to D$ is a double cover, $t$ is not a character of $D$, only $\chi$ is. But since $R(D)$ embeds in 
$R(D')$, the localized description of $\opk_{D}^\circ (Y)$ will be defined by the same divisibility conditions as that 
of $\opk_{D'}^\circ (Y)$, just taken in the subring $R(D)\subset R(D')$. 

\smallskip  

If $Y$ is a $G'$-fixed point, then $\opk_{D'}^\circ(Y)\simeq R(D')$.

If $Y=\mathbb{P}^1$, then $\opk_{D'}^\circ(Y)\simeq \{(f,g)\in R(D')^{\oplus 2}\,|\, f- g \equiv 0 \mod 1-\ee^{-\alpha}\}$, where $\alpha=2t$ is the positive root of $G'$.

For the cases (3) to (5), we shall obtain an explicit presentation of the equivariant $K$-theory rings by following Brion's description of the corresponding equivariant Chow groups \cite[Proposition 7.2]{brion-chow}.  Recall that the character $t$ identifies $D'$ with $\mathbb{G}_m$, as in Notation~\ref{n.fixed}, so $R(D') \isom \ZZ[\ee^{\pm t}]$.

For the projective plane $\mathbb{P}(V)$, with $V=\Sym^2\CC^2$, the weights of $D'$ acting on $V$ are $-2t$, $0$, and $2t$.  We denote by $x,y,z$ the corresponding $D'$-fixed points, so $x=[1,0,0]$, $y=[0,1,0]$, and $z=[0,0,1]$.  We make the identification $\opk_{D'}^\circ(\mathbb{P}(V)^{D'}) = R(D')^{\oplus 3}$, using this ordering of fixed points.

For $\mathbb{P}^1\times \mathbb{P}^1$ with the diagonal action of $G'=SL_2$, the torus $D'$ acts diagonally with weights $-t,t$ on each factor.  This action has exactly four fixed points, which we write as $x = ([1,0],[1,0])$, $y=([0,1],[1,0])$, $z=([1,0],[0,1])$, and $w=([0,1],[0,1])$, and identify $\opk_{D'}^\circ((\PP^1\times\PP^1)^{D'}) = R(D')^{\oplus 4}$ using this ordering.  

Finally, for a Hirzebruch surface $\mathbb{F}_n$ ($n\geq 1$) with ruling $\pi\colon \mathbb{F}_n\to \mathbb{P}^1$, there are exactly four $D'$-fixed points $x,y,z,w$, where $x,z$ (resp. $y,w$) are mapped to $0=[1,0]$ (resp. $\infty=[0,1]$) by $\pi$.  We assume that $x$ and $y$ lie in the $G'$-invariant section $C_+$ (with positive self-intersection), and that $z$ and $w$ lie in the negative $G'$-invariant section $C_-$.  With this ordering of the fixed points, we identify $\opk_{D'}^\circ(\mathbb{F}_n^{D'})$ with $R(D')^{\oplus 4}$.

\begin{theorem}\label{t.spherical-localize}
With notation as above, for $Y$ one of these three surfaces, the image of the  homomorphism $\iota_{D'}^* \colon \opk_{D'}^\circ(Y) \to \opk_{D'}^\circ(Y^{D'})$ is as follows.

\begin{enumerate}[itemsep=3pt]

\item ($Y=\PP(V)$.)  Triples $(f_x,f_y,f_z)$ such that
\begin{align*}
f_x-f_y\equiv f_y-f_z &\equiv 0 \mod (1-\ee^{-2t}), \\
f_x-f_z &\equiv 0 \mod (1-\ee^{-4t}), \\
\intertext{and}
f_x - \ee^{-2t}(1+\ee^{-2t})\,f_y + \ee^{-6t}\, f_z & \equiv 0 \mod (1-\ee^{-2t})(1-\ee^{-4t}).
\end{align*}

\item ($Y=\PP^1\times\PP^1$.) Quadruples $(f_x,f_y,f_z,f_w)$ such that
\begin{align*}
f_x- f_y \equiv f_x- f_z\equiv f_y- f_w \equiv f_z -  f_w &\equiv 0 \mod (1-\ee^{-2t})
\intertext{and}
f_x - \ee^{-2t}\,f_y - \ee^{-2t}\,f_z + \ee^{-4t}\,f_w &\equiv 0 \mod (1-\ee^{-2t})^2.
\end{align*}

\item ($Y=\mathbb{F}_n$.)  Quadruples $(f_x,f_y,f_z,f_w)$ such that
\begin{align*}
f_x- f_y\equiv  f_z- f_w  &\equiv 0 \mod (1-\ee^{-2t}), \\
f_x- f_z \equiv f_y- f_w   &\equiv 0 \mod (1-\ee^{-nt}), \\
\intertext{and}
f_x + \ee^{-(n+2)t}\,f_y - \ee^{-nt}\,f_z - \ee^{-2t}\,f_w &\equiv 0 \mod (1-\ee^{-2t})(1-\ee^{-nt})
\end{align*}
\end{enumerate}
\end{theorem}

\begin{proof}
The two-term conditions come from $T$-invariant curves, as in \eqref{e.divisible} above.  The three- and four-term conditions may similarly be deduced from the requirement
\[
 \sum_{p\in Y^{D'}} \eem_p(Y)\cdot f_p \in R(D').
\]

To write these out, one needs computations of the tangent weights at each fixed point.  For $\PP(V)$ and $\PP^1\times\PP^1$, these computations are standard, using the actions specified.  For $\mathbb{F}_n$, we consider it as the subvariety of $\PP^2 \times\PP^1$ defined by
\[
 \mathbb{F}_n = \{ ([a_0,a_1,a_2], [b_1,b_2]) \,|\, a_1b_1^n = a_2b_2^n\},
\]
with $D'$ acting by
\[
  \zeta\cdot ([a_0,a_1,a_2], [b_1,b_2]) = ([a_0,\zeta^{n}a_1,\zeta^{-n}a_2], [\zeta^{-1}b_1,\zeta b_2]).
\]
The weights on fixed points of $\mathbb{F}_n$ are as follows:
\[
\begin{array}{c|c}
{\rm Fixed \,point} & {\rm weights} \\
\hline
 x = ([0,0,1],[1,0]) & 2t, nt \\
 y = ([0,1,0],[0,1]) & -2t, -nt \\
 z = ([1,0,0],[1,0]) & 2t, -nt \\
 w = ([1,0,0],[0,1]) & -2t, nt
\end{array}
\]

Now the three-term relation for $\PP(V)$ comes from clearing denominators in the condition that
\[
 \frac{f_x}{(1-\ee^{-2t})(1-\ee^{-4t})} + \frac{f_y}{(1-\ee^{-2t})(1-\ee^{2t})} + \frac{f_z}{(1-\ee^{2t})(1-\ee^{4t})},
\]
belong to $R(D')$.  Similarly, the four-term relation for $\PP^1\times\PP^1$ and $\mathbb{F}_n$ come from requiring that
\[
 \frac{f_x}{(1-\ee^{-2t})^2} + \frac{f_y}{(1-\ee^{-2t})(1-\ee^{2t})} + \frac{f_z}{(1-\ee^{-2t})(1-\ee^{2t})} + \frac{f_w}{(1-\ee^{2t})^2}
\]
and
{\small
\[
 \frac{f_x}{(1-\ee^{-2t})(1-\ee^{-nt})} + \frac{f_y}{(1-\ee^{2t})(1-\ee^{nt})} + \frac{f_z}{(1-\ee^{-2t})(1-\ee^{nt})} + \frac{f_w}{(1-\ee^{2t})(1-\ee^{-nt})},
\]
}
respectively, belong to $R(D')$.

To see that the divisibility conditions are sufficient, one can use a Bia{\l}ynicki-Birula decomposition to produce an $R(D')$-linear basis of $K_{D'}^\circ(Y)$, and verify that the conditions guarantee a tuple may be expressed as a linear combination of such basis elements. 
We carry out this explicitly for the case $Y=\mathbb{F}_n$,   
and leave the other cases as exercises, since they can be checked in a similar way. 
We proceed inductively. For any $f_x\in R(D')$, the element $(f_x,f_x,f_x, f_x)=f_x\cdot(1,1,1,1)$ 
is certainly in the image of $\iota_D^*$, because $(1,1,1,1)=\iota_D^*([\mathcal{O}_{\mathbb{F}_n}])$. 
To see that $(f_x,f_y,f_z,f_w)$ is in the image, it suffices 
to show that $(0,f_y-f_x,f_z-f_x,f_w-f_x)$ is in the image; that is, we may assume the first entry is zero. 
By the divisibility conditions, we can write such an element as $(0,(1-e^{-2t})g_y,g_z,g_w)$. 
Now note that $-e^{-2t}g_y[\mathcal{O}_{\pi^{-1}(\infty)}]\in K_{D'}^\circ(\mathbb{F}_n)$ 
restricts to $(0,(1-e^{-2t})g_y,0,(1-e^{-2t})g_y)$, and by subtracting this, we reduce to the case where the 
first two entries are zero. So, again by the divisibility conditions, 
it suffices to prove that $(0,0,(1-e^{-nt})h_z,h_w)$ lies in the image. 
Next, observe that the element $-e^{-nt}h_z[\mathcal{O}_{C_-}]\in K_{D'}^\circ(\mathbb{F}_n)$ restricts to $(0,0,(1-e^{-nt})h_z,-e^{-nt}(1-e^{-nt})h_z)$, and by subtracting this, we can reduce finally to the 
case where the first three entries are zero. Thus, by the divisibility conditions, it suffices to prove that $(0,0,0,(1-e^{-2t})(1-e^{-nt})s_w)$ lies in the image. But this is the restriction 
of $-s_w e^{-2t}[\mathcal{O}_{\{w\}}]\in K_{D'}^\circ(\mathbb{F}_n)$. 

In summary, we have shown that any element $(f_x,f_y,f_z,f_w)\in R(D')^{\oplus 4}$ that satisfies 
the divisibility conditions belongs to the linear span of 
the images of the classes 
$[\mathcal{O}_{\mathbb{F}_n}]$, $[\mathcal{O}_{\pi^{-1}(\infty)}]$, $[\mathcal{O}_{C_-}]$, and $[\mathcal{O}_{\{w\}}]$. 
Since these classes freely generate $K_{D'}^\circ (\mathbb{F}_n)$, the result follows.     
\end{proof}

\begin{remark}
The conditions presented here complete the description claimed in \cite[Theorem~1.1]{ekt-sph}, where the three- and four-term relations are missing. To see that these relations are indeed necessary, consider the case 
$Y=\PP(V)$.  
Then $K_{D'}^\circ (\PP(V))$ is freely generated by the classes of the structure sheaves of 
the point $z$, 
the line $(yz)$ and the whole $\PP(V)$. 
These classes restrict respectively to  
\[
(0,0,(1-\ee^{-2t})(1-\ee^{-4t})),\;\, (0,1-\ee^{-2t},\,1-\ee^{-4t}), \;\, (1,1,1).
\] 
Certainly they satisfy the divisibility relations.  
However, the triple $(0,0,1-e^{-4t})$ satisfies the two-term conditions   
of \cite[Theorem~1.1]{ekt-sph}, 
but it does not lie in the span of those basis elements.  
\end{remark}

\smallskip

Next, we consider the case when $Y$ is the normal surface $P_n$ obtained by contracting the unique section $C_{-}$ of negative self-intersection in $\mathbb{F}_n$, as in item (\ref{Pn}) above.  For $n>1$, this surface is singular.  We use the fact that the map $q:\mathbb{F}_n\to P_n$, which contracts $C_-$ to a fixed point, is an (equivariant) envelope to calculate $\opk_D^\circ(Y)$ from $\opk_D^\circ(\mathbb{F}_n)$ using the Kimura sequence.

\begin{lemma}
Let $P_n=\mathbb{F}_n/C_-$ be the weighted projective plane obtained by contracting the unique section $C_-$ of negative self-intersection in $\mathbb{F}_n$, so that the fixed points of $P_n$ are identified with $x$, $y$, $z$.  Then the image of $\opk_{D'}^\circ(P_n)\to R(D')^{\oplus 3}$ consists of all triples $(f_x,f_y,f_z)$ such that
\begin{align*}
f_x-f_z \equiv f_y-f_z &\equiv 0 \mod 1-e^{-nt},\\
f_x-f_y &\equiv 0 \mod 1-e^{-2t},\\
\intertext{and}
f_x + \ee^{-(n+2)t} f_y - (\ee^{-2t}+\ee^{nt}) f_z &\equiv 0 \mod (1-e^{-nt})(1-\ee^{-2t}).
\end{align*}
\end{lemma}

\begin{proof}
Note that $\pi\colon\mathbb{F}_n\to P_n$ is an envelope.  We write $(\mathbb{F}_n)^{D'}=\{x',y',z',w'\}$, so that $x'\mapsto x$, $y'\mapsto y$, and $z',w'\mapsto z$.  By the Kimura sequence, an element $(f_{x'},f_{y'},f_{z'},f_{w'})\in \opk_{D'}^\circ((\mathbb{F}_n)^{D'})$ lies in the image of $\pi^*$ if and only if it satisfies the relations defining $\opk_{D'}^\circ(\mathbb{F}_n)$, together with the extra relation $f_{z'}=f_{w'}$ (which accounts for the fact that $C_-$ is collapsed to a point in $P_n$).  The relations from Theorem~\ref{t.spherical-localize}(3) reduce to those asserted here.
\end{proof}

Finally, we consider the case when the surface with a double curve obtained by identifying the sections $C_+$ and $C_-$ in $\mathbb{F}_n$ appears as an irreducible component of $X^H$.

\begin{lemma}
Let $\mathcal{K}_n$ be the non-projective algebraic surface with an ordinary double curve obtained by identifying the curves $C_+$ and $C_-$ of the surface $\mathbb{F}_n$, so that the fixed points of $\mathcal{K}_n$ are identified with $x$, $y$.  Then the image of $\opk_{D'}^\circ(\mathcal{K}_n)\to R(D')^{\oplus 2}$ consists of all $(f_x,f_y)$ such that $f_x-f_y\equiv 0 \mod 1-e^{-2t}$.
\end{lemma}

\begin{proof}
Identifying the curves $C_+$ and $C_-$ of $\mathbb{F}_n$ implies that we identify the fixed points $x$ with $z$, and $y$ with $w$.  Using the Kimura sequences, we see that the relations describing $\opk_{D'}^\circ(\mathbb{F}_n)$ reduce, after this identification, to the asserted ones.
\end{proof}

Summarizing our previous results, in view of Theorem~\ref{loc.thm.cs} and Lemma~\ref{l.irr}, yields the main result of this section.  It is an extension of Brion's work on the equivariant Chow rings of complete nonsingular spherical varieties (\cite[Theorem 7.3]{brion-chow}) to the equivariant operational $K$-theory of possibly singular complete spherical varieties.  For the corresponding statement in rational equivariant operational Chow cohomology see \cite{go.tlin}.

\begin{theorem}\label{opk.sph.thm}
Let $X$ be a complete spherical $G$-variety.  
The image of the injective map
\[
\iota^*\colon\opk^\circ_T(X)\to \opk^\circ_T(X^T)
\]
consists of all families
$(f_x)_{x\in X^T}\in \bigoplus_{x\in X^T}R(T)$
satisfying the following relations:
\begin{enumerate}[itemsep=3pt]
\item $f_x-f_y\equiv 0  \mod (1-\ee^{-\chi})$, whenever $x,y$ are connected by a $T$-invariant curve with weight $\chi$.

\item 
$f_x - \ee^{-\chi}(1+\ee^{-\chi})\,f_y + \ee^{-3\chi}\, f_z  \equiv 0 \mod (1-\ee^{-\chi})(1-\ee^{-2\chi})$ whenever $\chi$ is a root, and $x$, $y$, $z$ lie in an irreducible component of $X^{\ker(\chi)^\circ}$ whose normalization is $SL_2$-equivariantly isomorphic to $\mathbb{P}(V)$.

\item 
$f_x - \ee^{-\chi}\,f_y - \ee^{-\chi}\,f_z + \ee^{-2\chi}\,f_w \equiv 0 \mod (1-\ee^{-\chi})^2$, whenever $\chi$ is a root, and $x$, $y$, $z$, $w$ lie in an irreducible component of $X^{\ker(\chi)^\circ}$ whose normalization is 
$SL_2$-equivariantly isomorphic to $\mathbb{P}^1\times \mathbb{P}^1$.

\item  
$f_x + \ee^{-(n+2)\chi/2}\,f_y - \ee^{-n\chi/2}\,f_z + \ee^{-\chi}\,f_w \equiv 0 \mod (1-\ee^{-\chi})(1-\ee^{-n\chi/2})$, where $\chi$ is a root, and $x$, $y$, $z$, $w$ lie in an irreducible component of $X^{\ker(\chi)^\circ}$ whose normalization is $SL_2$-equivariantly isomorphic to the Hirzebruch surface $\mathbb{F}_n$ for $n\geq 1$.  (The case of odd $n$ is possible only when $\chi/2$ is a weight of $T$.)

\item
$f_x + \ee^{-(n+2)\chi/2} f_y - (\ee^{-\chi}+\ee^{n\chi/2}) f_z \equiv 0 \mod (1-e^{-n\chi/2})(1-\ee^{-\chi})$, where $\chi$ is a root, and $x$, $y$, $z$
lie in an irreducible component of $X^{\ker(\chi)^\circ}$ whose normalization is $SL_2$-equivariantly isomorphic to the weighted projective plane $P_n$ obtained by contracting the curve $C_-$ of negative self-intersection in $\mathbb{F}_n$. 
\qed
\end{enumerate} 
\end{theorem}

\appendix
\section{Change-of-groups homomorphisms}\label{s.A.change-groups}

The goal of this appendix is to construct a natural change-of-groups homomorphism in operational $K$-theory.  
We start by briefly recalling some basic facts in equivariant $K$-theory.  See \cite{thomason-alg} and \cite{merkurjev} for details.

Let $G$ be an algebraic group.  
Recall that a $G$-scheme is a scheme $X$ together with an action morphism
$a\colon G\times X\to X$  that satisfies the usual identities \cite{thomason-alg}.  
Equivalently, a $G$-scheme is a scheme $X$ together with an action of $G(S)$ on the set $X(S)$ for each scheme $S$, functorially in $S$.  
A $G$-module $M$ over $X$ is a quasi-coherent $\OO_X$-module $M$
together with an isomorphism of $\OO_{G\times X}$-modules
\[
 \rho=\rho_M\colon  a^*(M)\xrightarrow{\sim} p_2^*(M)
\]
(where $p_2\colon G\times X\to X$ is the projection),
satisfying the cocycle condition
\[
  p_{23}^*(\rho)\circ ({\id}_G\times a)^*(\rho)= (m\times {\id}_X)^*(\rho),
\]
where $p_{23}\colon G\times G\times X \to G\times X$ is the projection and
$m\colon G\times G \to G$ is the product morphism.  
A morphism of $G$-modules is a morphism of modules $\alpha\colon M\to N$ such that $\rho_N\circ a^*(\alpha)=p_2^*(\alpha)\circ \rho_M$.  We write $\mathcal{M}(G,X)$ for the abelian category of coherent $G$-modules
over a $G$-scheme $X$, and set $K_\circ^G(X)$ to be the Grothendieck
group of this category.

A flat morphism $f\colon X\to Y$ of $G$-schemes
induces an exact functor
\[
 \mathcal{M}(G,Y)\to \mathcal{M}(G,X), \;\;\; M\mapsto f^*(M),
\]
and therefore defines the pull-back homomorphism
$f^*\colon K_\circ^G(Y)\to K_\circ^G(X)$.

Let $\pi\colon H\to G$ be a homomorphism of algebraic groups,
and let $X$ be a $G$-scheme.
The composition
\[
   H\times X \xrightarrow{\pi\times {\id}_X} G\times X \xrightarrow{a} X
\]
makes $X$ an $H$-scheme.
Given a $G$-module $M$ with the $G$-module structure defined by an isomorphism $\rho$,
we can introduce an $H$-module structure on $M$ via $(\pi\times {\id}_X)^*(\rho)$.
Thus, we obtain an exact functor
\[
  {\rm Res}_\pi\colon \mathcal{M}(G,X)\to \mathcal{M}(H,X)
\]
inducing the {\em restriction} homomorphism
\[
  {\rm res}_\pi\colon K_\circ^G(X)\to K_\circ^H(X).
\]
If $H$ is a subgroup of $G$, we write ${\rm res}_{G/H}$ for the restriction homomorphism ${\rm res}_{\pi}$, where $\pi\colon H\hookrightarrow G$ is the inclusion.

\medskip

Let $G$ and $H$ be algebraic groups, and let $f\colon X\to Y$ be a $G\times H$-morphism
of $G\times H$-varieties.  Assume that $f$ is a $G$-torsor (in particular, $G$ acts trivially on $Y$).  Let $M$ be a coherent $H$-module over $Y$.  Then $f^*(M)$ has a structure of a coherent $G\times H$-module over $X$ given by $p^*(\rho_M)$, where $p$ is the composition of the projection $G\times H\times X\to H\times X$ and the morphism $({\id}_H\times f)\colon H\times X \to H\times Y$.  Thus, there is an exact functor
\[
  f^0\colon \mathcal{M}(H,Y)\to \mathcal{M}(G\times H,X),\;\;\; M\mapsto p^*(M).
\]

\begin{proposition}[\protect{\cite[Proposition~2.3]{merkurjev}}] \label{Gtorsor.prop}
The functor $f^0$ is an equivalence of categories.  In particular, the homomorphism $K^H_\circ(Y)\to K^{G\times H}_\circ(X)$, induced by $f^0$, is an isomorphism.
\end{proposition}

\begin{corollary}[\protect{\cite[Corollary~2.5]{merkurjev}}]\label{ind.cor}
Let $G$ be an algebraic group and let $H\subset G$ be a subgroup.  For every $G$-scheme $X$, there is a natural isomorphism
\[
  K_\circ^G(X\times (G/H))\simeq K^H_\circ(X).
\]
\end{corollary}

In particular, by taking $X$ a point, we get $R(H)\simeq K_\circ^G(G/H)$.  On the other hand, by applying Proposition~\ref{Gtorsor.prop} to the $H$-torsor $G\to G/H$, we get $K_\circ(G/H)\simeq K_\circ^H(G)$.

\medskip

We will prove a version of Proposition~\ref{Gtorsor.prop} in equivariant operational $K$-theory.  For technical reasons, we must confine our statements to tori.  Let $T_1$ and $T_2$ be tori, and write $T=T_1\times T_2$.  Suppose $X \to Y$ is a $T$-equivariant morphism, with $T_1$ acting trivially.  Then we have
\[
  \opk^\circ_{T}(X \to Y) \isom R(T_1)\otimes \opk^\circ_{T_2}(X \to Y),
\]
by \cite[Corollary~5.6]{ap}.  (In \cite{ap}, this is only stated for the contravariant theory, but the proof is the same for the full bivariant theory.)  Using this identification, there is a pullback homomorphism
\[
  \opk^\circ_{T_2}(X \to Y) \to \opk^\circ_{T_1\times T_2}(X \to Y),
\]
sending $c \mapsto 1\otimes c$.

Next we consider a fiber diagram
\begin{diagram}
Z & \rTo & W  \\
\dTo^{\tilde{f}} &  & \dTo_f \\
X & \rTo & Y,
\end{diagram}
of $T_1\times T_2$-equivariant morphisms, still assuming $T_1$ acts trivially on $X$ and $Y$.  In this context, we have a homomorphism
\[
 f^*\colon \opk^\circ_{T_2}(X\to Y) \to \opk^\circ_{T_1\times T_2}(Z \to W),
\]
defined by composing the above change-of-groups pullback with the usual pullback across fiber squares.

\begin{proposition}\label{op.Gtorsor.prop}
In the above setup, assume $W \to Y$ is a $T_1$-torsor, so $Z \to X$ is also a $T_1$-torsor.  Then the pullback $f^*\colon \opk^\circ_{T_2}(X\to Y)\to \opk^\circ_{T_1\times T_2}(Z\to W)$ is an isomorphism.
\end{proposition}

\begin{proof}
If $Y$ is smooth, then so is $W$, and we have natural Poincar\'e isomorphisms
\[
  \opk^\circ_{T_2}(X\to Y) \isom K_\circ^{T_2}(X) \quad \text{and} \quad  \opk^\circ_{T}(Z\to W) \isom K_\circ^{T}(Z).
\]
Our claim follows by applying Proposition~\ref{Gtorsor.prop} with $G=T_1$ and $H=T_2$.

We will apply the second Kimura sequence (see \S\ref{ss.chow-K}, \eqref{e.kimura2}) to reduce to the case where $Y$ is smooth.  Choose a birational equivariant envelope $\tilde{Y} \to Y$ with $\tilde{Y}$ smooth.  Let $\tilde{W}\to W$ be the pullback, so $\tilde{W}\to\tilde{Y}$ is again a $T_1$ torsor; in particular, $\tilde{W}\to W$ is also a birational envelope, and $\tilde{W}$ is also smooth.

With notation as in \S\ref{ss.chow-K}, let $B\subseteq Y$ and $E\subseteq \tilde{Y}$ be such that the map $\tilde{Y}\to Y$ restricts to an isomorphism $\tilde{Y}\setminus E \xrightarrow{\sim} Y\setminus B$.  Let $A\subseteq X$ and $D\subseteq \tilde{X}$ be the pullbacks to $X$ and $\tilde{X}$; similarly, let $A'\subseteq Z$, $B'\subseteq W$, $D'\subseteq \tilde{Z}$, and $E'\subseteq \tilde{W}$ be the respective pullbacks.  The Kimura sequences for the birational envelopes $\tilde{Y} \to Y$ and $\tilde{W}\to W$ fit together in a diagram
{\small
\begin{diagram}
  0 & \rTo & \opk_{T_2}^\circ (X \to Y) & \rTo & \opk_{T_2}^\circ( \tilde{X} \to \tilde{Y}) \oplus \opk_{T_2}^\circ(A \to B) & \rTo & \opk_{T_2}^\circ (D \to E ) \\
   &       & \dTo  &    & \dTo  &   &  \dTo  \\
  0 & \rTo & \opk_{T}^\circ (Z \to W) & \rTo & \opk_{T}^\circ( \tilde{Z} \to \tilde{W}) \oplus \opk_{T}^\circ(A' \to B') & \rTo & \opk_{T}^\circ (D' \to E' )  ,
\end{diagram}
}

\noindent
with exact rows.  The middle and rightmost vertical arrows are isomorphisms, by induction on dimension and the smooth case, so it follows that the leftmost vertical arrow is an isomorphism, as desired.
\end{proof}

Finally, let $T$ be a torus, with a subtorus $T'\subseteq T$.  Let $X$ be a $T$-scheme.  As an application of Proposition~\ref{op.Gtorsor.prop}, one constructs a natural {\em restriction homomorphism}
\[
  {\rm res}_{T/T'}\colon \opk^\circ_T(X)\to \opk^\circ_{T'}(X).
\]
Indeed, using the proposition and arguing as in Corollary \ref{ind.cor}, we have a natural isomorphism
\[
 \opk_T^\circ(X\times (T/T'))\simeq \opk_{T'}^\circ(X).
\]
The restriction homomorphism is the composition of this isomorphism with pullback along the first projection $X \times T/T' \to X$.

\section{A Grothendieck transformation from algebraic to operational $K$-theory \\ \medskip \emph{by G.~Vezzosi}}  \label{app:vezzosi}

We describe a generalization of operational $K$-theory in derived algebraic geometry and use this, together with properties of the truncation functor to ordinary schemes, to prove the following theorem.
\begin{theorem} \label{thm:Groth-transform}
There is a Grothendieck transformation from the algebraic $K$-theory of $f$-perfect complexes to bivariant operational $K$-theory, taking an $f$-perfect complex $\shfE$ to the corresponding Gysin homomorphisms $f^{\shfE} \in \opk(f)$.
\end{theorem}
\noindent The main difficulty is showing that the Gysin homomorphisms $f^{\shfE}$ satisfy the bivariant axioms (A1) and (A2) in \cite[Definition~4.1]{ap} required to be elements of $\opk(f)$. Indeed, the relevant diagrams do not commute at the level of sheaves on schemes, and we must show that they do commute at the level of $K$-theory.  The key new observations are that the derived analogues of these diagrams do commute, up to homotopy, at the level of complexes of sheaves on derived schemes, and the natural functors between schemes and derived schemes preserve $K$-theory. In particular, while the statement of the theorem is purely about the $K$-theory of morphisms of schemes, the proof uses derived algebraic geometry in an essential way. For background in derived algebraic geometry, we refer the reader to \cite{toen-higher, toen-dag, tv}.  

Throughout, we work over a fixed ground field and assume that all derived schemes are quasi-compact, separated and weakly of finite type, meaning that their truncations are quasi-comapct, separated and of finite type.  All relevant functors on complexes of sheaves on derived schemes, such as push-forward, pullback, and tensor product, are implicitely derived.

\smallskip

Let $\Sch$ denote the category of schemes and let $\dSch$ be the homotopy category of the model category of derived schemes.  Recall that the inclusion $i \colon \Sch \rightarrow \dSch$ is fully faithful and left adjoint to the truncation functor $t_0 \colon \dSch \rightarrow \Sch$ \cite{tv}.  When no confusion seems possible, we will write simply $X$ or $f$, rather than $i(X)$ or $i(f)$, to denote the derived object or morphism associated to an object or morphsim in $\Sch$. Since $t_0$ is right adjoint to $i$, whenever we have a homotopy cartesian square in $\dSch$, 
\begin{diagram}
 \fX' & \rTo & \fY' \\
 \dTo & &  \dTo \\
 \fX & \rTo & \fY,
\end{diagram}
the induced diagram
\begin{diagram}
 t_0 \fX' & \rTo & t_0 \fY' \\
 \dTo & &  \dTo \\
 t_0 \fX & \rTo & t_0 \fY,
\end{diagram}
is cartesian in $\Sch$.

\smallskip

Let $\fX$ be a derived scheme.  Let $\mathsf{QCoh} (\fX)$ be the $\infty$-category of quasi-coherent complexes on $\fX$, as in \cite[\S3.1]{toen-higher}.  We define $\mathsf{Coh}(\fX)$ to be the full $\infty$-subcategory of $\mathsf{QCoh}(\fX)$ whose objects $\mathscr{E}$ have coherent cohomology over $t_0 \fX$ that vanishes in all but finitely many degrees. We write $D_{\coh}(\fX)$ for the homotopy category of $\mathsf{Coh}(\fX)$.  It is a sub-triangulated category of the homotopy category $D_{\qcoh}(\fX)$ of $\mathsf{QCoh}(\fX)$.  

Let $K_\circ(\fX)$ be the Grothendieck group of the triangulated category $D_{\coh}(\fX)$.

\begin{definition} \label{def:properflat}
A morphism of derived schemes $f \colon \fX \to \fY$ is 
\begin{itemize}
\item \emph{proper}, respectively, a \emph{closed immersion}, if $t_0 f$ is so;
\item a \emph{regular embedding} if it is a closed immersion and quasi-smooth (i.e. it is locally of finite presentation and the relative cotangent complex $\mathbb{L}_f$ is of Tor-amplitude $\leq 1$); 
\item \emph{flat} if it is flat as in \cite{tv} (more precisely, see \cite{tv} Definition 2.2.2.3 (2), Proposition 2.2.2.5 (4), for derived affine schemes, and Lemma 2.2.3.4 for the case of arbitrary derived schemes).
\end{itemize}  \end{definition}

\begin{remark}
Note that if $f \colon \fX \to \fY$ is flat, then its truncation $t_0 f  \colon t_0\fX \to t_0\fY$ is flat as a map of usual schemes. 
A map between underived schemes is a regular embedding if and only if it is a regular embedding between derived schemes according to Definition \ref{def:properflat}
(see, e.g. \cite[2.3.6]{KhRy}). A crucial property of regular embeddings between derived schemes is that it is stable under arbitrary (homotopy) pullbacks; such a property is false for regular embeddings of underived schemes and usual scheme theoretic pullbacks. Note, however, that in general, the truncation of a regular embedding between derived schemes might not be a classical regular embedding. 
\end{remark}

\begin{definition}
For a morphism of derived schemes $f\colon \fX \to \fY$, we define $\opk^{\textrm{der}}(f)$ exactly as in \cite[Definition~4.1]{ap}, where all schemes are replaced by derived schemes, pullbacks are replaced by homotopy pullbacks, and proper morphisms, flat morphisms, and regular embeddings are as defined above.
\end{definition}

We start by proving two lemmas that are derived generalizations of \cite[Lemmas~3.1-3.2]{ap}.  Recall that, throughout this appendix, all push forwards, pullbacks, and tensor products of complexes of sheaves on derived schemes are derived.

Let $f \colon \fX \to \fY$ be a morphism in $\dSch$, and let $\shfE$ be an $f$-perfect complex on $\fX$.  For each homotopy cartesian square
\begin{diagram}
\fX' & \rTo^{f' }& \fY' \\
 \dTo^{g'} & & \dTo_g \\
 \fX & \rTo^{f} & \fY,
 \end{diagram}
 we define a Gysin pullback $f^{\shfE}\colon \mathsf{Coh}( \fY') \to \mathsf{Coh} (\fX')$ by setting
 \[
 f^{\shfE}(\shfF) = g'^* \shfE \otimes_{\OO_{X'}} f'^* \shfF,
 \]
for $\shfF \in \mathsf{Coh}( \fY')$ \footnote{This is well defined: the derived pull-back always maps $\mathsf{Coh}^{-}$ to itself, therefore $g'^* \shfE \otimes_{\OO_{X'}} f'^* \shfF$ is in $\mathsf{Coh}^{-}(\fX')$, and it is actually inside $\mathsf{Coh}(\fX')$ because \cite[Exp. III, Cor. 4.7.2]{sga6} holds in derived algebraic geometry without the Tor-independence hypothesis (note that the cartesian square used to define $f^{\shfE}$ is a homotopy cartesian square).  }.  We also write $f^\shfE$ for the induced map $K_\circ(\fY') \to K_\circ(\fX')$
\[
 f^{\shfE}[\shfF] = [g^* \shfE \otimes_{\OO_{X'}} f'^* \shfF].
 \]

\begin{lemma} \label{lem:tower}
Consider a tower of homotopy cartesian squares in $\dSch$, 
\begin{diagram}
 \fX'' & \rTo^{f''} & \fY'' \\
 \dTo^{h'} & &  \dTo_{h} \\
 \fX' & \rTo^{f' }& \fY' \\
 \dTo^{g'} & & \dTo_g \\
 \fX & \rTo^{f} & \fY,
\end{diagram}
and suppose $h$ is proper.  Let $\shfE$ be an $f$-perfect complex on $\fX$.  Then  $$f^{\shfE} \!\! \circ h_* = h'_* \circ f^{\shfE} \!,$$ as maps $K_\circ(\fY'') \to K_\circ(\fX')$.
\end{lemma}

\begin{proof}
Let $\shfF \in \mathsf{Coh} (\fY'')$.  We have
\[
f^{\shfE} h_*[\shfF] = f^{\shfE}[h_* \shfF] = [g'^*\shfE \otimes_{\OO_{X'}} f'^* h_* \shfF].
\]
By the base-change formula \cite[Proposition~1.4]{toen-intersections}, we have\footnote{This is another step where we use $\dSch$ in a crucial way; the analogous statement does not hold for cartesian diagrams in $\Sch$, without further hypotheses.} $$f'^*h_*\shfF \cong h'_* f''^*\shfF,$$ and hence
\begin{equation} \label{eq:dot}
f^{\shfE} h_* [\shfF] = [g'^* \shfE \otimes_{\OO_{X'}} h'_*f''^* \shfF].
\end{equation}

On the other hand, we have:
\[
h'_*f^{\shfE}[\shfF] = h'_*[h'^*g'^*\shfE \otimes_{\OO_{X''}} f''^*\shfF] = [h'_*(h'^*g'^*\shfE \otimes_{\OO_{X''}} f''^* \shfF].
\]

Applying the projection formula, we get
\[
h'_*(h'^*g'^*\shfE \otimes_{\OO_{X''}} f''^*\shfF) \cong g'^* \shfE \otimes_{\OO_{X'}} h'_*f''^*\shfF,
\]
and hence
\begin{equation}\label{eq:dotdot}
h'_*f^{\shfE}[\shfF] = [g'^*\shfE \otimes_{\OO_{X'}} h'_* f''^* \shfF].
\end{equation}
Comparing \eqref{eq:dot} and \eqref{eq:dotdot} gives $f^{\shfE} h_* [\shfF] = h'_* f^{\shfE} [\shfF]$, as required.
\end{proof}

\begin{lemma} \label{lem:3squares}
Consider the following diagram in $\dSch$, with homotopy cartesian squares:
\begin{diagram}
 \fX'' & \rTo^{f''} & \fY'' & \rTo^{u'}  &  \fZ''  \\
\dTo^{h''} & & \dTo^{h'}    &       & \dTo_{h} \\
 \fX' & \rTo^{f'} & \fY'    & \rTo^u  &  \fZ' \\
\dTo^{g'} & & \dTo_g \\
\fX  & \rTo^{f}  & \fY.
\end{diagram}
Suppose $\shfE$ is $f$-perfect and $\shfV$ is $h$-perfect.  Then $f^{\shfE} \circ h^{\shfV} = h^{\shfV} \circ f^{\shfE}$ as maps $K_\circ(\fY') \to K_\circ(\fX'')$.
\end{lemma}

\begin{proof}
Let $\xi \in \mathsf{Coh}(\fY')$. Then
\begin{align*}
f^{\shfE} \circ h^{\shfV} [\xi] & = [h''^*g'^*\shfE \otimes_{\OO_{X''}} f''^*(h'^*\xi \otimes_{\OO_{Y''}} u'^* \shfV]; \\
& = [h''^*g'^*\shfE \otimes_{\OO_{X''}} f''^*u'^* \shfV \otimes_{\OO_{X''}} f''^*h'^*\xi].
\end{align*}
Similarly,
\begin{align*}
h^{\shfV} \circ f^{\shfE} [ \xi ] & = [ f''^* u'^* \shfV \otimes_{\OO_{X''}} h''^* (g'^* \shfE \otimes_{\OO_{X'}}f'^* \xi)]; \\
& = [f''^*u'^* \shfV \otimes_{\OO_{X''}} h''^*g'^* \shfE \otimes_{\OO_{X''}} h''^*f'^*\xi].
\end{align*}
The lemma follows, since $h''^*f'^* \xi \cong f''^* h'^* \xi$.
\end{proof}


A crucial step in the proof of Theorem \ref{thm:Groth-transform} is the following 

\begin{proposition} \label{prop:derived-isom}
Let $f \colon \fX \to \fY$ be a morphism in $\dSch$. Then there is a canonical injective morphism of groups 
\[
\alpha: \opk^{\textrm{der}}(f) \longrightarrow \opk(t_0 f).
\]
\end{proposition}

\begin{proof}
We begin by observing that, for any derived scheme $\fX$, the natural map
\begin{equation}\label{heart}
j_* \colon K_\circ(t_0 \fX) \rightarrow K_\circ (\fX)
\end{equation}
is an isomorphism, where $j \colon t_0 \fX \to \fX$ is the closed immersion of the truncation into the derived scheme.  See \cite[\S3.1, p.~193]{toen-dag}.

\smallskip

Let $c = \{ c_g \} \in \opk^{\textrm{der}}(f)$, and let 
\begin{diagram}
 X' & \rTo & Y' \\
 \dTo & &  \dTo_{h} \\
 t_0 \fX & \rTo^{t_0 f} & t_0 \fY,
\end{diagram}
be cartesian in $\Sch$.  Consider the homotopy cartesian square in $\dSch$
\begin{diagram}
 \fX' & \rTo & Y' \\
 \dTo & &  \dTo_{j \circ h} \\
 \fX & \rTo^f & \fY,
\end{diagram}
where the righthand vertical arrow is the composition of $h$ with the closed embedding $j\colon t_0 \fY \to \fY$.

By applying the truncation functor, we obtain a cartesian square in $\Sch$
\begin{diagram}
 t_0 \fX' & \rTo & Y' \\
 \dTo & &  \dTo_{h} \\
 t_0 \fX & \rTo^f & t_0 \fY,
\end{diagram}
Therefore, $t_0\fX' \cong X'$.  We then set (using (\ref{heart})) $\alpha(c)_h = c_{j \circ h}$.  Using lemmas \ref{lem:tower} and \ref{lem:3squares}, together with axioms (A1) and (A2) for $\opk^{\textrm{der}}(f)$, one may check that, indeed, $\alpha(c) \in \opk(t_0 f)$, i.e. $\alpha(c)$ verifies axioms (A1) and (A2) for $\opk(t_0 f)$. We leave these details to the reader. Since $\alpha$ obviously preserves the sum of two morphisms, we have obtained a well defined group homomorphism $\alpha \colon \opk^{\textrm{der}}(f) \to \opk(t_0 f)$.

\smallskip

We now show that $\alpha$ is injective.  Suppose $c, c' \in \opk^{\textrm{der}}(f)$ satisfy $\alpha(c) = \alpha(c')$.  Set notation $\alpha(c) = \{ c^\alpha_g \}$ and $\alpha(c') = \{ c'^{\alpha}_g \}$. Suppose $g\colon Y' \to t_0 \fY$ and 
\begin{diagram}
 X' & \rTo & Y' \\
 \dTo & &  \dTo_g \\
 t_0 \fX & \rTo^{t_0 f} & t_0 \fY,
\end{diagram}
is cartesian in $\Sch$.  Then $c_g^\alpha$ and $c'^{\alpha}_g$ are defined in terms of the homotopy cartesian square
\begin{diagram}
 \fX' & \rTo & Y' \\
 \dTo & &  \dTo_{j \circ g} \\
 \fX & \rTo^f & \fY,
\end{diagram}
by setting $c^\alpha_g$ = $c_{j \circ g}$ and $c'^{\alpha}_g = c'_{j \circ g}$.  We are assuming that $c^\alpha_g = c'^{\alpha}_g$ for all relevant arrows $g$ in $\Sch$ and must show that $c_h = c'_h$ for all relevant arrows $h$ in $\dSch$.  

Let $h\colon \fY' \to \fY$ in $\dSch$, and suppose
\begin{diagram}
\fX' & \rTo & \fY' \\
 \dTo & &  \dTo_h \\
 \fX & \rTo^f & \fY,
\end{diagram}
is homotopy cartesian.  Consider the cartesian diagram in $\Sch$ obtained from this by truncation.  We know, by hypothesis, that $c^\alpha_{t_0 h} = c'^{\alpha}_{t_0 h}$, i.e., that $c_{j \circ t_0 h} = c'_{j \circ t_0 h}$.  Now observe that, by functoriality of $t_0$, the diagram
\begin{diagram}
 t_0 \fY' & \rTo^{t_0 h} & t_0 \fY \\
 \dTo^{j'} & &  \dTo_j \\
 \fY' & \rTo^{h} & \fY,
\end{diagram}
is commutative, and hence, by forming the homotopy cartesian square
\begin{diagram}
 \fX'' & \rTo & t_0 \fY' \\
 \dTo & &  \dTo_{h \circ j'} \\
 \fX & \rTo^{f}&  \fY,
\end{diagram}
in $\dSch$ (with the same $\fX'$), we deduce  $c_{h \circ j'} = c'_{h \circ j'}$ (note that $t_0\fX'' \simeq t_0\fX'$, hence $K_\circ(\fX'') \simeq K_\circ(\fX')$ by (\ref{heart})). We complete the proof that $\alpha$ is injective by showing that, if $c, c' \in \opk^{\textrm{der}}(f)$ satisfy $c_{h \circ j'} = c'_{h \circ j'}$ for all $h\colon \fY' \to \fY$, then $c = c'$.\\
In order to do this, we consider a tower of homotopy cartesian squares
\begin{diagram}
 \fX'' & \rTo & t_0 \fY' \\
 \dTo^\rho & &  \dTo_{j'} \\
 \fX' & \rTo^f & \fY' \\
 \dTo & & \dTo_h \\
 \fX & \rTo^{f} & \fY.
\end{diagram}
Since $j' : t_0 \fY' \hookrightarrow \fY'$ is proper, the property (A1) in the definition of $\opk^{\textrm{der}}(f)$ (\cite[Definition~4.1]{ap}) tells us that the inner and outer squares of
\begin{diagram}
 K_\circ(\fY') & \pile{ \rTo^{c'_h} \\ \rTo_{c_h}} & K_\circ \fX' \\
 \uTo^{j'_*} & &  \uTo_{\rho_*} \\
 K_\circ(t_0 \fY') & \pile{ \rTo^{c_{h \circ j'}} \\ \rTo_{c'_{h \circ j'}}} & K_\circ(\fX''),
\end{diagram}
commute (separately). The lefthand vertical arrow $j'_*$ is an isomorphism, so the equality $c'_{h \circ j'} = c_{h \circ j'}$ implies $c'_h = c_h$, as claimed. This concludes the proof of Proposition \ref{prop:derived-isom}.
\end{proof}

\begin{remark} The truncation of a regular embedding is not, in general, a classical regular embedding, so our proof does not extend to show the map $\alpha$ is an isomorphism (as we claimed in a previous version of the paper). We thank the careful referee for addressing this point.
However, even if, for the purposes of this Appendix, injectivity of $\alpha$ is sufficient, T.~Annala in a recent preprint (\cite{annala2}) gave a proof that $\alpha$ is indeed bijective.
\end{remark}

\begin{proof}[Proof of Theorem~\ref{thm:Groth-transform}]
Let $f \colon X \to Y$ be a morphism in $\Sch$, and let $\shfE$ be an $f$-perfect complex.   Apply the functor $i \colon \Sch \to \dSch$, view $\shfE$ as an $i(f)$-perfect complex on $i(X)$, and consider the collection of Gysin homorphisms $i(f)^\shfE \colon K_\circ(\fY') \to K_\circ(\fX')$, for homotopy cartesian squares
\begin{diagram}
\fX' & \rTo & \fY' \\
 \dTo & & \dTo \\
 i(X) & \rTo^{i(f)} & i(Y),
 \end{diagram}
 in $\dSch$.  Lemmas~\ref{lem:tower} and \ref{lem:3squares} show that these Gysin homorphisms satisfy the bivariant axioms (A1) and (A2) from \cite[Definition~4.1]{ap}, respectively, and hence give rise to an element $i(f)^\shfE \in \opk^{\textrm{der}}(i(f))$.  We then obtain the required Grothendieck transformation by taking $[\shfE]$ to the image of $i(f)^\shfE$ in $\opk(f)$, under the the morphism $\alpha$ in Proposition~\ref{prop:derived-isom} (note that $t_0(f)=f$, here).
\end{proof}

We conclude with a result on composition of Gysin maps associated to $f$-perfect complexes in operational $K$-theory of derived schemes. The special case where $f$ is a regular embedding, $g$ is smooth, and $\shfV = \OO_Y$ is the derived analogue of \cite[Lemma~3.3]{ap}.

\begin{proposition} \label{prop:lem3.3}
Let $f\colon \fX \to \fY$ and $g \colon \fY \to \fZ$ be morphisms in $\dSch$.  Let $\shfE$ be $f$-perfect, and let $\shfV$ be $g$-perfect.  Then $f^{\shfE} \circ g^{\shfV} = (g \circ f)^{\shfE \otimes f^* \shfV}$, provided that $\shfE \otimes f^* \shfV$ is $(g \circ f)$-perfect.
\end{proposition}

\begin{proof}
Consider the following diagram, with homotopy cartesian squares:
\begin{diagram}
 \fX' & \rTo^{f'} & \fY' & \rTo^{g'} & \fZ' \\
\dTo_{h''} & & \dTo_{h'}   &     & \dTo_{h} \\
 \fX & \rTo^{f} & \fY   &  \rTo^g  &  \fZ.
 \end{diagram}
Let $\shfF \in \mathsf{Coh}(\fZ')$.  We have 
\begin{align*}
f^{\shfE} \circ g^{\shfV} [\shfF] & = [h''^* \shfE \otimes_{\OO_{X'}} f'^*(h'^*\shfV \otimes_{\OO_{Y}} g'^* \shfF)] \\
& = [h''^* \shfE \otimes_{\OO_{X'}} f'^*h'^* \shfV \otimes_{\OO_{X'}} f'^*g'^*\shfF].
\end{align*}
Similarly,
\begin{align*}
(g \circ f)^{\shfE \otimes f^*\shfV} [\shfF] & = [h''^*(\shfE \otimes_{\OO_X} f^* \shfV) \otimes f'^*g'^*\shfF] \\
&= [h''^* \shfE \otimes_{\OO_{X'}} h''^*f^* \shfV \otimes_{\OO_{X'}} f'^*g'^* \shfF] \\
\end{align*}
The lemma follows, since $f'^*h'^* \shfV \cong h''^* f^* \shfV$.
\end{proof}

Combining Propositions~\ref{prop:derived-isom} and \ref{prop:lem3.3}, we deduce the following corollary for canonical orientations of morphisms in $\Sch$.  This generalizes \cite[Lemma~4.2]{ap}, and solves a problem raised in loc. cit.

\begin{corollary}
If $f\colon X \to Y$ and $g\colon Y \to Z$ are morphisms of finite $\mathrm{Tor}$-dimension in $\Sch$ then $f^! \circ g^! = (g \circ f)^!$.
\end{corollary}

\begin{proof}
Since $f$ has finite $\tor$-dimension, the structure sheaf $\OO_X$ is $f$-perfect, and $f^! = f^{\OO_X}$, and similarly for $g$.  Applying Proposition~\ref{prop:lem3.3} to the morphisms $i(f)$ and $i(g)$ in $\dSch$, with $\shfE = \OO_{i(X)}$ and $\shfV = \OO_{i(Y)}$ shows that $i(g \circ f)^! = i(f)^! \circ i(g)^!$. The corollary follows, using Proposition~\ref{prop:derived-isom} to pass from $\opk^{\textrm{der}}(f)$ to $\opk(f)$ (note that $f=t_0(f)$, here). \end{proof}



\end{document}